\documentclass[11pt]{amsart}
\usepackage{amsmath,amssymb,amsthm}
\usepackage{graphicx}
\usepackage{booktabs}
\usepackage[margin=1.1in]{geometry}
\usepackage[colorlinks=true,linkcolor=blue,citecolor=blue]{hyperref}

\newtheorem{theorem}{Theorem}[section]
\newtheorem{lemma}[theorem]{Lemma}
\newtheorem{proposition}[theorem]{Proposition}
\newtheorem{corollary}[theorem]{Corollary}
\theoremstyle{definition}
\newtheorem{assumption}[theorem]{Assumption}
\newtheorem{rulenv}[theorem]{Rule}
\newtheorem{remark}[theorem]{Remark}

\newcommand{\X}{X}
\newcommand{\D}{\mathcal{D}}
\newcommand{\N}{\mathcal{N}}
\newcommand{\R}{\mathcal{R}}
\newcommand{\ip}[2]{\langle #1,#2\rangle}
\newcommand{\dg}{\dagger}

\title[A noise-level-free parameter choice for iterated Lavrentiev regularization]{A noise-level-free parameter-choice rule for iterated Lavrentiev regularization of nonlinear potential monotone ill-posed problems}

\author{Ankit Singh}
\address{Department of Applied Science and Humanities, SVKM's College of Engineering, Shirpur, India}
\email{ankit1vaishali@gmail.com}

\author{Hemant Sharma}
\address{Department of Mathematics, Indian Institute of Information Technology, Design and Manufacturing Kancheepuram, Chennai, India}
\email{sharmahemant39@gmail.com}
\thanks{Corresponding author: Hemant Sharma (\texttt{sharmahemant39@gmail.com})}




\begin{document}

\begin{abstract}
We consider nonlinear ill-posed operator equations $F(x)=y$ in a Hilbert space, where $F$ is monotone with a Lipschitz continuous Fr\'echet derivative that is self-adjoint near the solution $x^\dg$ (a locally \emph{potential} operator), and only noisy data $y^\delta$ with unknown noise norm $\delta:=\|y-y^\delta\|$ is available. For the nonstationary iterated Lavrentiev regularization
\[
x_{k+1}^\delta=x_k^\delta-\bigl(F'(x_k^\delta)+\alpha_kI\bigr)^{-1}\bigl(F(x_k^\delta)-y^\delta+\alpha_k(x_k^\delta-x_0)\bigr),
\]
all classical stopping rules require $\delta$. We analyze a \emph{noise-level-free} parameter-choice rule of retrospective type, selecting a minimizer $k_*$ of $\theta(k,y^\delta)=\|F(x_k^\delta)-y^\delta\|^2/\alpha_k^2$ over a finite computed window. Under a standard noise condition (necessary in view of the Bakushinskii veto) the rule is well defined and $\alpha_{k_*}^2\ge c\,\delta/\|x_0-x^\dg\|$. If moreover $x_0-x^\dg\in\R(F'(x^\dg))$ and $k_*\ge1$, we prove $\|x_{k_*}^\delta-x^\dg\|=O(\sqrt{\delta_*}+\sqrt{\delta})$ with the computable $\delta_*=\|F(x_{k_*}^\delta)-y^\delta\|$ satisfying $\kappa\delta\le\delta_*\le C\sqrt\delta$, yielding the baseline rate $O(\delta^{1/4})$ along admissible growing-horizon families, and $O(\sqrt\delta)$ whenever $\delta_*=O(\delta)$; no range-invariance or tangential cone conditions are imposed. Rates under H\"older source conditions with $p\ge1$ (including saturation) follow; the case $p<1$, the removal of the self-adjointness assumption, and a full oracle inequality are identified as open problems. Numerical experiments in one, two, and three dimensions with monotone operators $F(x)=A^*\varphi(Ax)$, organized as independently recomputed matched-model validation runs complemented by clearly separated model-mismatch stress tests, compare the rule with two noise-level-based a~posteriori rules and support the theory within its stated scope.
\end{abstract}

\maketitle

\section{Introduction}\label{sec:intro}

\subsection{The problem and the method}\label{subsec:problem}

We wish to solve the nonlinear operator equation
\begin{equation}\label{eq:main}
F(x)=y,
\end{equation}
where $F:\D(F)\subset\X\to\X$ is a nonlinear operator on a Hilbert space $\X$ with inner product $\ip{\cdot}{\cdot}$ and norm $\|\cdot\|$. Throughout, $F$ is assumed to be \emph{monotone},
\[
\ip{F(x)-F(\tilde x)}{x-\tilde x}\ \ge\ 0 \qquad \text{for all } x,\tilde x\in\D(F),
\]
and \emph{Fr\'echet differentiable} in a neighborhood of a solution $x^\dg$ of \eqref{eq:main}, with a Lipschitz continuous derivative $F'$; the precise conditions are collected in Assumption~\ref{ass:nonlinearity} below. Equation \eqref{eq:main} is \emph{ill-posed}: its solution does not depend continuously on the right-hand side. This matters because in practice the exact data $y$ is not accessible; we are only given noisy data $y^\delta\ne y$, and we define
\begin{equation}\label{eq:noise}
\delta:=\|y-y^\delta\| ,
\end{equation}
the exact noise norm. Without regularization, even arbitrarily small data errors can produce arbitrarily large solution errors.

Two comments on the convention \eqref{eq:noise} are in order. First, $\delta$ is in general unknown; this causes no difficulty because neither the method nor the stopping rule below uses it. Second, the exact norm (rather than an upper bound $\bar\delta$) is essential for the lower bounds in the analysis: the noise condition $\|F(x_k^\delta)-y^\delta\|\ge\kappa\|y-y^\delta\|$ yields quantitative information only in terms of the exact norm, whereas upper estimates remain valid with $\delta$ replaced by any known bound $\bar\delta\ge\delta$.

For monotone problems, a natural regularization is the Lavrentiev method \cite{taut02,taut04}, which replaces \eqref{eq:main} by the singularly perturbed equation
\begin{equation}\label{eq:lav}
F(x)-y^\delta+\alpha(x-x_0)=0 ,
\end{equation}
with regularization parameter $\alpha>0$ and initial guess $x_0$ for $x^\dg$; unlike Tikhonov regularization \cite{binder,engl-hanke,engl-kunisch,jin-hou97,jin-hou99,scherzer}, it requires neither adjoints nor a variational formulation, which makes it attractive computationally; see \cite{hofmann,mahale-nair09,mahale-nair13,nair-taut,semenova} for its analysis under various parameter choices. In this paper we study the iterated (nonstationary) form of the method introduced by Bakushinsky and Smirnova \cite{bak-smir07}: starting from $x_0^\delta:=x_0$, one computes
\begin{equation}\label{eq:iter}
x_{k+1}^\delta = x_k^\delta - \bigl(A_k^\delta+\alpha_k I\bigr)^{-1}\bigl(F(x_k^\delta)-y^\delta+\alpha_k(x_k^\delta-x_0)\bigr),
\qquad A_k^\delta:=F'(x_k^\delta),
\end{equation}
where $(\alpha_k)$ is a prescribed sequence of regularization parameters satisfying
\begin{equation}\label{eq:alphas}
\alpha_k>0,\qquad 1\le\frac{\alpha_k}{\alpha_{k+1}}\le\mu,\qquad \lim_{k\to\infty}\alpha_k=0,
\end{equation}
for some $\mu>1$ (for instance, $\alpha_k=\alpha_0q^k$ with $q\in(0,1)$ and $\mu=1/q$). Each step of \eqref{eq:iter} requires only one linearization and one linear solve, so the method is simple to implement. Its regularizing effect, however, is controlled entirely by the stopping index: stopping too early leaves the approximation dominated by the regularization error, while iterating too long lets the iterates fit the noise and diverge from $x^\dg$ (semiconvergence). The central question addressed in this paper is therefore:
\begin{center}
\emph{How should the iteration \eqref{eq:iter} be stopped when the noise level $\delta$ is not known?}
\end{center}

\subsection{Stopping rules with and without the noise level}\label{subsec:rules}

The classical answers to the stopping problem all presuppose knowledge of $\delta$. Bakushinsky and Smirnova \cite{bak-smir07} used a generalized discrepancy principle \cite{bak-smir05,bak-smir06}, choosing the first index $k_\delta$ with
\begin{equation}\label{eq:gdp}
\|F(x_{k_\delta}^\delta)-y^\delta\|^2\le\tau\delta\le\|F(x_k^\delta)-y^\delta\|^2,\qquad 0\le k<k_\delta,
\end{equation}
for a constant $\tau>1$, and proved convergence $x_{k_\delta}^\delta\to x^\dg$ as $\delta\to0$ (without rates). Mahale and Nair \cite{mahale-nair09} instead stopped at the first index satisfying the weighted residual criterion
\begin{equation}\label{eq:mnrule}
\bigl\|\alpha_{k_\delta}\bigl(F'(x_{k_\delta}^\delta)+\alpha_{k_\delta}I\bigr)^{-1}\bigl(F(x_{k_\delta}^\delta)-y^\delta\bigr)\bigr\|\le\tau\delta,
\end{equation}
for some $\tau\ge1$, and obtained order-optimal error estimates, but under structural nonlinearity assumptions on $F$ that are, in general, stronger than a Lipschitz condition on $F'$. In \cite{mahale-singh}, Mahale and the present author showed that the analysis of \eqref{eq:iter} with the rule \eqref{eq:mnrule} can be carried out under monotonicity, Fr\'echet differentiability, and only a Lipschitz condition on $F'$.

Both \eqref{eq:gdp} and \eqref{eq:mnrule}, like essentially all a~priori and a~posteriori rules, take $\delta$ as input. In many applications the noise level is unknown, or only a rough (and possibly wrong) estimate is available, and an incorrect $\delta$ can severely degrade the reconstruction; our numerical experiments in Section~\ref{sec:numerics} illustrate this sensitivity. \emph{Heuristic} (noise-level-free) rules address this difficulty by selecting the stopping index from computable quantities alone. The idea goes back to the quasi-optimality principle \cite{tikhonov-glasko} and to Hanke and Raus \cite{hanke-raus}; see also \cite{hansen-oleary,jin-raus,wahba} and \cite[Chapter~4]{engl-hanke}. It is well known that no heuristic rule can yield a convergent regularization method for arbitrary noise (the Bakushinskii veto); a rigorous analysis is possible under a condition restricting the admissible noise, and we adopt the noise condition introduced by Jin and Wang \cite{jin-wang}, who analyzed a heuristic rule of this type for the iteratively regularized Gauss--Newton method.

The rule studied in this paper is of the same type, adapted to the Lavrentiev setting: we stop \eqref{eq:iter} at a minimizer $k_*$ of the computable functional
\[
\theta(k,y^\delta)=\frac{\|F(x_k^\delta)-y^\delta\|^2}{\alpha_k^2}
\]
(see Rule~\ref{rule:heuristic} in Section~\ref{sec:prelim} for the precise formulation and the computable noise surrogate $\delta_*:=\|F(x_{k_*}^\delta)-y^\delta\|$ appearing in our error estimates). Nothing in this rule requires $\delta$.

\subsection{What is known, and what this paper adds}\label{subsec:contrib}

The following table summarizes the state of the art for the iterated Lavrentiev method \eqref{eq:iter}; ``strong nonlinearity condition'' refers to structural range-invariance conditions of the type $(F'(x)-F'(u))v=F'(u)g(x,u,v)$ with $\|g(x,u,v)\|\le k_0\|v\|\|x-u\|$ used in \cite{mahale-nair09}, and ``Lipschitz condition'' to \eqref{eq:lipschitz} below.

\begin{center}
\begin{tabular}{lll}
\toprule
Reference & Stopping rule & Nonlinearity condition on $F$\\
\midrule
Bakushinsky--Smirnova \cite{bak-smir07} & discrepancy type \eqref{eq:gdp}, needs $\delta$ & Lipschitz (no rates)\\
Mahale--Nair \cite{mahale-nair09} & weighted rule \eqref{eq:mnrule}, needs $\delta$ & strong\\
Mahale--Singh \cite{mahale-singh} & weighted rule \eqref{eq:mnrule}, needs $\delta$ & Lipschitz\\
Mahale--Singh \cite{mahale-singh-etna} & heuristic, noise-level-free & strong\\
\textbf{this paper} & \textbf{heuristic, noise-level-free} & \textbf{Lipschitz (+ self-adjoint $F'$)}\\
\bottomrule
\end{tabular}
\end{center}

\smallskip
In other words, noise-level-free stopping for \eqref{eq:iter} has so far been justified only under the strong nonlinearity condition \cite{mahale-singh-etna}, while the Lipschitz condition has so far required knowledge of $\delta$ \cite{mahale-singh}. This paper narrows that gap: the heuristic rule is analyzed under monotonicity, Fr\'echet differentiability, self-adjointness of the derivative near the solution ($F$ locally a potential operator; Remark~\ref{rem:selfadjoint}), and a Lipschitz condition on $F'$, with no range-invariance or tangential cone conditions imposed. The self-adjointness requirement, needed by our spectral-calculus arguments and satisfied by all test operators in Section~\ref{sec:numerics}, marks the honest boundary of what is proved; its removal is posed as an open problem. Related error estimates by the author for a simplified iteratively regularized Gauss--Newton method under a general source condition can be found in \cite{mahale-singh-kumar}.

Specifically, the contributions of this paper are the following.
\begin{itemize}
\item[(i)] \emph{Well-definedness.} Under the noise condition of \cite{jin-wang} (Assumption~\ref{ass:noise}), the heuristic rule selects a finite index $k_*$, and the corresponding regularization parameter admits the lower bound $\alpha_{k_*}^2\ge c\,\delta/\|x_0-x^\dg\|$ with an explicit constant $c\in(0,1)$ (Lemma~\ref{lem:kstar-bound}).
\item[(ii)] \emph{Error estimate.} If the initial error satisfies the range condition $x_0-x^\dg\in\R(F'(x^\dg))$, then
$\|x_{k_*}^\delta-x^\dg\|\le C(\sqrt{\delta_*}+\sqrt{\delta})$ with the computable surrogate $\delta_*$ (Theorem~\ref{thm:rate-range}).
\item[(iii)] \emph{Interpretation of $\delta_*$, source conditions, and honest limitations.} We prove the two-sided bound $\kappa\delta\le\delta_*\le C\sqrt\delta$, yielding the baseline rate $O(\delta^{1/4})$ under the stated assumptions, and $O(\sqrt\delta)$ whenever $\delta_*=O(\delta)$, a regime not observed asymptotically in our experiments (Corollary~\ref{cor:deltastar}, Remark~\ref{rem:growing-horizon}); we derive the rates for H\"older source conditions with $p\ge1$, explain the resulting saturation, and identify $p<1$ as open (Theorem~\ref{thm:hoelder}, Remark~\ref{rem:plessone}). We also prove an intermediate estimate towards an oracle inequality and identify the two obstructions preventing its completion (Proposition~\ref{thm:oracle-pre}, Subsection~\ref{subsec:obstruction}).
\item[(iv)] \emph{Numerical validation.} We construct a general class of genuinely monotone test operators $F(x)=A^*\varphi(Ax)$, which are valid for arbitrary bounded smoothing operators $A$ (in particular non-self-adjoint ones), and compare, on one-, two-, and three-dimensional examples, the heuristic rule with the discrepancy principle and with the weighted rule \eqref{eq:mnrule}, documenting the strengths and limitations of each (Section~\ref{sec:numerics}).
\end{itemize}

\subsection{Outline and notation}\label{subsec:outline}

Section~\ref{sec:prelim} formulates the heuristic rule precisely and collects the assumptions. Section~\ref{sec:wd} establishes well-definedness of the iteration and of the rule and proves the error estimate of order $O(\sqrt{\delta_*}+\sqrt{\delta})$ under the range condition; Subsection~\ref{subsec:noisefree} treats the noise-free iteration and provides a stability estimate. Section~\ref{sec:rates} interprets the surrogate $\delta_*$, derives the rates under H\"older source conditions with $p\ge1$, and documents the obstructions to a full oracle inequality. Section~\ref{sec:numerics} presents the numerical experiments, and Section~\ref{sec:concl} concludes. Throughout, $x^\dg$ denotes a solution of \eqref{eq:main}, $e_0:=x_0-x^\dg$ the initial error, $A:=F'(x^\dg)$, $A_k^\delta:=F'(x_k^\delta)$, and $\N(\cdot)$, $\R(\cdot)$ the null space and range of an operator; further notation is introduced in Section~\ref{sec:prelim}.

\section{Preliminaries, assumptions, and the heuristic rule}\label{sec:prelim}

Given noisy data $y^\delta$, let $(x_k^\delta)$ be generated by \eqref{eq:iter} and define
\begin{equation}\label{eq:kinfty}
k_\infty:=k_\infty(y^\delta):=\max\{k:\ x_l^\delta\in\D(F)\ \text{for all } 0\le l\le k\},
\end{equation}
i.e., $k_\infty$ is the largest index for which all iterates generated by \eqref{eq:iter} remain in $\D(F)$; it may be finite or infinite. We can now formulate the stopping rule studied in this paper.

\begin{rulenv}[Heuristic parameter-choice rule]\label{rule:heuristic}
Fix a finite computational horizon $K\in\mathbb{N}$, compute the iterates $x_0^\delta,\dots,x_{K_\infty}^\delta$ with $K_\infty:=\min\{K,k_\infty\}$, define
\begin{equation}\label{eq:theta}
\theta(k,y^\delta):=\frac{\|F(x_k^\delta)-y^\delta\|^2}{\alpha_k^2},\qquad 0\le k\le K_\infty,
\end{equation}
and choose the stopping index $k_*=k_*(y^\delta)$ such that
\begin{equation}\label{eq:kstar}
k_*\in\arg\min\bigl\{\theta(k,y^\delta):\ k=0,1,\dots,K_\infty\bigr\}.
\end{equation}
The output of the method is $x_{k_*}^\delta$. We also set
\begin{equation}\label{eq:deltastar}
\delta_*:=\|F(x_{k_*}^\delta)-y^\delta\|,
\end{equation}
which serves as a computable surrogate for the (unknown) noise level and appears in the error estimates below.
\end{rulenv}

\begin{remark}
Rule~\ref{rule:heuristic} differs from the rule of Jin and Wang \cite{jin-wang} in the power of $\alpha_k$ appearing in the denominator of \eqref{eq:theta}; this normalization is adapted to the Lavrentiev setting and to the a~priori choice \eqref{eq:khat} used in the analysis below.
\end{remark}

\begin{remark}[Implementability]\label{rem:retrospective}
Rule~\ref{rule:heuristic} is a \emph{retrospective} parameter-choice rule, not an online stopping criterion: the trajectory is computed up to the horizon $K$, understood as a \emph{computational budget}, the values $\theta(k,y^\delta)$ are recorded, and the minimizer over the window is selected. The results below are conditional on the window containing the relevant comparison indices: Lemma~\ref{lem:kstar-bound} and Theorem~\ref{thm:rate-range} require $K_\infty\ge\hat k$ from \eqref{eq:khat}, and Corollary~\ref{cor:deltastar} requires $K_\infty\ge k_*^\delta$ from \eqref{eq:kdstar}. These conditions are not certifiable from the data alone: $\hat k$ and $k_*^\delta$ grow as $\delta$ decreases, so a known upper bound $\bar\delta\ge\delta$ does not yield a sufficient horizon, while a positive lower bound on $\delta$ would, but is generally unavailable. For the geometric choice $\alpha_k=\alpha_0q^k$, $q\in(0,1)$, both indices grow only like $\log(1/\delta)$, so moderate budgets suffice in practice; the general condition \eqref{eq:alphas} alone does not imply such growth (it admits $\alpha_k=\alpha_0/(k+1)$). As a practical, though not certified, indicator, Assumption~\ref{ass:noise} forces $\theta(k,y^\delta)\ge\kappa^2\delta^2/\alpha_k^2\to\infty$, so a sustained increase of $\theta$ past its running minimum suggests that the minimizer has been passed; in the experiments of Section~\ref{sec:numerics}, $\theta$ was strictly increasing over the final third of the window in 1D and 2D, asymptotically constant in 3D (Appendix~\ref{app:horizon}), and enlarging $K$ changed $k_*$ in no run. A proved first-local-minimum or finite-look-ahead criterion is not claimed.
\end{remark}

As is well known, no heuristic rule can lead to a convergent regularization method for all admissible noise (this is the content of the Bakushinskii veto; see \cite{engl-hanke}). Following \cite{jin-wang}, we therefore restrict the class of admissible noise by the following condition.

\begin{assumption}\label{ass:noise}
There exists a constant $\kappa$ with $0<\kappa<1$ such that
\begin{equation}\label{eq:noisecond}
\|y^\delta-\tilde y\|\ \ge\ \kappa\,\|y^\delta-y\|
\end{equation}
for every $\tilde y\in\{F(x):\ x\in S(y^\delta)\}$, where $S(y^\delta):=\{x_k^\delta:0\le k\le k_\infty\}$ and $(x_k^\delta)$ is generated by \eqref{eq:iter} with data $y^\delta$.
\end{assumption}

Assumption~\ref{ass:noise} expresses that a certain portion of the data error cannot be attained by the computed approximations; it was used in \cite{jin-wang} in the error analysis of the iteratively regularized Gauss--Newton method. The following lemma establishes the well-definedness of Rule~\ref{rule:heuristic}.

\begin{lemma}\label{lem:wd-rule}
(i) A minimizer $k_*$ satisfying \eqref{eq:kstar} always exists, since the minimization in Rule~\ref{rule:heuristic} extends over the finite computed window $\{0,\dots,K_\infty\}$.

(ii) Suppose in addition that Assumption~\ref{ass:noise} holds and $y^\delta\ne y$. Then even the infinite-horizon minimum $\inf_{0\le k\le k_\infty}\theta(k,y^\delta)$ is attained at some finite index.
\end{lemma}

\begin{proof}
Part (i) is trivial. For (ii), let $k_\infty=\infty$ (otherwise there is nothing to prove). By Assumption~\ref{ass:noise} and \eqref{eq:alphas},
\[
\theta(k,y^\delta)=\frac{\|F(x_k^\delta)-y^\delta\|^2}{\alpha_k^2}\ \ge\ \frac{\kappa^2\|y-y^\delta\|^2}{\alpha_k^2}\ \longrightarrow\ \infty
\qquad\text{as } k\to\infty,
\]
since $\|y-y^\delta\|>0$ and $\alpha_k\to0$. Hence $\theta(\cdot,y^\delta)$ attains its minimum at some finite index.

We emphasize that (ii) is a statement about the infinite sequence: it guarantees that a sufficiently large window captures the infinite-horizon minimizer, but it does not provide a computable bound for how large the window must be; see Remark~\ref{rem:retrospective}.
\end{proof}

We next state the nonlinearity conditions used throughout the paper.

\begin{assumption}\label{ass:nonlinearity}
\begin{itemize}
\item[(i)] $F$ is monotone on $\D(F)$.
\item[(ii)] There exists $r>0$ such that $B_r(x^\dg):=\{x\in\X:\|x-x^\dg\|<r\}\subseteq\D(F)$ and $F$ is Fr\'echet differentiable on $B_r(x^\dg)$. Moreover, $F$ is properly scaled in the sense that
\begin{equation}\label{eq:scaling}
\|F'(x)\|\le\alpha_0\qquad\text{for all } x\in B_r(x^\dg).
\end{equation}
\item[(iii)] There exists $L>0$ such that
\begin{equation}\label{eq:lipschitz}
\|F'(x)-F'(z)\|\le L\|x-z\|\qquad\text{for all } x,z\in B_r(x^\dg).
\end{equation}
\item[(iv)] For every $x\in B_r(x^\dg)$ the Fr\'echet derivative is self-adjoint, $F'(x)=F'(x)^*$; by monotonicity of $F$ it is then automatically nonnegative, $F'(x)\ge0$.
\end{itemize}
\end{assumption}

\begin{remark}\label{rem:selfadjoint}
Condition (iv) deserves comment, since it restricts the class of monotone operators covered by our analysis. (a) Operators with symmetric derivative are exactly the (local) \emph{potential operators}: if $F'$ is symmetric on the ball $B_r(x^\dg)$, then $F$ coincides there with the gradient of the convex functional $J(x)=\int_0^1\ip{F(x^\dg+t(x-x^\dg))}{x-x^\dg}\,dt$ (up to a constant); conversely, gradients of convex $C^2$-functionals are monotone with self-adjoint nonnegative derivatives. (b) The general class of test operators $F(x)=A^*\varphi(Ax)$ constructed in Section~\ref{sec:numerics} satisfies (iv) for arbitrary bounded linear $A$, since $F'(x)=A^*\varphi'(Ax)A$ is always self-adjoint and nonnegative, even when $A$ itself is non-self-adjoint. (c) Condition (iv) is used in the proofs of Proposition~\ref{prop:DE} and Lemma~\ref{lem:Bk-noisy} (and thus in Theorem~\ref{thm:rate-range}), where spectral-calculus arguments such as $B_k\le\mu B_{k+1}$ and the symmetric rearrangement of an inner product are employed; these arguments are not available for general non-normal accretive derivatives. Whether the results of this paper persist for monotone operators with genuinely non-self-adjoint derivatives is an open question.
\end{remark}

Condition \eqref{eq:lipschitz} is the Lipschitz condition referred to in the title, with Lipschitz constant $L$. We emphasize that, in addition, monotonicity and Fr\'echet differentiability of $F$ are assumed throughout; by Assumption~\ref{ass:nonlinearity}(i) and (iv), $F'(x)$ is a self-adjoint nonnegative operator, so that $(F'(x)+\alpha I)^{-1}$ exists for every $\alpha>0$ and, by the spectral theorem,
\begin{equation}\label{eq:resolvent}
\|(F'(x)+\alpha I)^{-1}\|\le\frac1\alpha,\qquad \|(F'(x)+\alpha I)^{-1}F'(x)\|\le1,\qquad x\in B_r(x^\dg).
\end{equation}
The scaling condition \eqref{eq:scaling} is a normalization: replacing $F$, $y$, $y^\delta$ and $\delta$ by $cF$, $cy$, $cy^\delta$ and $c\delta$ for a suitable constant $c>0$ (which preserves monotonicity), one may always achieve \eqref{eq:scaling}.

For the convergence rate analysis we shall use source conditions of H\"older type.

\begin{assumption}\label{ass:source}
There exist $p>0$ and $v\in\N(A)^\perp$, where $A:=F'(x^\dg)$ and $\N(A)$ denotes the null space of $A$, such that
\begin{equation}\label{eq:source}
x_0-x^\dg=A^p v .
\end{equation}
For non-integer $p\in(0,1)$ the fractional power $A^p$ of the self-adjoint nonnegative operator $A$ (Assumption~\ref{ass:nonlinearity}(iv)) is defined by the spectral theorem, $A^p=\int\lambda^p\,dE_\lambda$; equivalently, by the Balakrishnan formula
\[
A^p v:=\frac{\sin(p\pi)}{\pi}\int_0^\infty s^{p-1}(A+sI)^{-1}Av\,ds ,
\]
and for general $p>0$ by composition with integer powers of $A$.
\end{assumption}

\begin{remark}\label{rem:p-usage}
We record explicitly which value of $p$ is used where. The analysis of Sections~\ref{sec:wd} and~\ref{sec:rates} is carried out under the \emph{range condition}
\begin{equation}\label{eq:range}
x_0-x^\dg=F'(x^\dg)v\quad\text{for some } v\in\N(F'(x^\dg))^\perp,
\end{equation}
i.e., under Assumption~\ref{ass:source} with $p=1$. Convergence rates under Assumption~\ref{ass:source} with $p\ge1$ are discussed in Theorem~\ref{thm:hoelder}; the case $p\in(0,1)$ is not covered by the present analysis, see Remark~\ref{rem:plessone}.
\end{remark}

For convenience we introduce the notation
\begin{gather*}
r_\alpha(\lambda)=\alpha(\alpha+\lambda)^{-1},\qquad
g_\alpha(\lambda)=(\alpha+\lambda)^{-1},\qquad
e_k^\delta=x_k^\delta-x^\dg,\\
e_0=x_0-x^\dg,\qquad
A=F'(x^\dg),\qquad A_k^\delta=F'(x_k^\delta).
\end{gather*}

\section{Well-definedness of the method and error estimates}\label{sec:wd}

In this section we show that the iteration \eqref{eq:iter} and Rule~\ref{rule:heuristic} are well defined and derive error estimates under the range condition \eqref{eq:range}. We first recall a standard consequence of the Lipschitz condition \eqref{eq:lipschitz}.

\begin{proposition}[{\cite[Chapter~10]{engl-hanke}}]\label{prop:taylor}
Let $x,u\in B_r(x^\dg)$. Then
\begin{equation}\label{eq:taylor}
\|F(x)-F(u)-F'(u)(x-u)\|\le\frac{L}{2}\,\|x-u\|^2 .
\end{equation}
\end{proposition}

The next technical lemma will be used repeatedly.

\begin{lemma}[{\cite{blaschke}}]\label{lem:recurrence}
Let $a,b,c>0$ satisfy $4ac\le(1-b)^2$ and set
\[
\theta:=\frac{(1-b)-\sqrt{(1-b)^2-4ac}}{2c}.
\]
If $(\theta_k)$ is a sequence of positive reals with
$\theta_{k+1}\le a+b\theta_k+c\theta_k^2$ and $\theta_0\le\theta$,
then $\theta_k\le\theta$ for all $k\ge0$.
\end{lemma}

For the analysis we fix, assuming $\alpha_0>\delta/\|e_0\|$, an integer $\hat k$ such that
\begin{equation}\label{eq:khat}
\alpha_{\hat k}^2\ \le\ \frac{\delta}{\|e_0\|}\ <\ \alpha_k^2,\qquad 0\le k<\hat k .
\end{equation}

\begin{lemma}\label{lem:wd-iter}
Let \eqref{eq:alphas}, \eqref{eq:khat}, Assumption~\ref{ass:nonlinearity} and the range condition \eqref{eq:range} hold. Assume furthermore that
\[
L\mu\|v\|+\mu\sqrt{2L\|v\|(1+\alpha_0)}\ \le\ 1
\]
and that
\begin{equation}\label{eq:smallness1}
\|e_0\|\Bigl((1-L\mu\|v\|)+\sqrt{(1-L\mu\|v\|)^2-2L\mu^2\|v\|(1+\alpha_0)}\Bigr)\ \le\ 2\alpha_0\mu\|v\|(1+\alpha_0).
\end{equation}
Then for all $0\le k\le\hat k$,
\begin{equation}\label{eq:errbound1}
\|x_k^\delta-x^\dg\|\ \le\ C\,\alpha_k\|v\|,\qquad
C:=\frac{2\mu(1+\alpha_0)}{1-L\mu\|v\|}.
\end{equation}
Moreover, if $r>C\alpha_0\|v\|$, then $x_k^\delta\in B_r(x^\dg)$ for all $0\le k\le\hat k$ and
\begin{equation}\label{eq:residual-khat}
\|F(x_{\hat k}^\delta)-y^\delta\|\ \le\ \tau\,\delta,
\end{equation}
where
\[
\tau:=(1+CL\|v\|)+\Bigl((1+C_1)\mu^2+\frac{L}{2}C^2\|v\|\Bigr)\frac{\|v\|}{\|e_0\|},
\]
\[
C_1:=\Bigl(2C+\frac{C^2}{2}\Bigr)L\|v\|+\Bigl(C^2+\frac{C^3}{2}\Bigr)L^2\|v\|^2 .
\]
\end{lemma}

\begin{proof}
We argue by induction on $k$. For $k=0$, \eqref{eq:errbound1} follows from \eqref{eq:smallness1}. Assume \eqref{eq:errbound1} holds for some $0\le k<\hat k$. From \eqref{eq:iter},
\begin{align}
x_{k+1}^\delta-x^\dg
&=(\alpha_kI+A_k^\delta)^{-1}\bigl\{(\alpha_kI+A_k^\delta)e_k^\delta-\bigl(F(x_k^\delta)-y^\delta+\alpha_k(x_k^\delta-x_0)\bigr)\bigr\}\notag\\
&=(\alpha_kI+A_k^\delta)^{-1}\bigl\{\alpha_k e_0-\bigl(F(x_k^\delta)-y-A_k^\delta e_k^\delta\bigr)+(y^\delta-y)\bigr\}\notag\\
&=\alpha_k(\alpha_kI+A)^{-1}e_0
+\bigl(\alpha_k(\alpha_kI+A_k^\delta)^{-1}-\alpha_k(\alpha_kI+A)^{-1}\bigr)e_0\notag\\
&\qquad-(\alpha_kI+A_k^\delta)^{-1}\bigl(F(x_k^\delta)-y-A_k^\delta e_k^\delta\bigr)
+(\alpha_kI+A_k^\delta)^{-1}(y^\delta-y).\label{eq:decomp}
\end{align}
Using the identity
$\alpha_k(\alpha_kI+A_k^\delta)^{-1}-\alpha_k(\alpha_kI+A)^{-1}
=\alpha_k(\alpha_kI+A_k^\delta)^{-1}(A-A_k^\delta)(\alpha_kI+A)^{-1}$,
the range condition $e_0=Av$, \eqref{eq:resolvent} and \eqref{eq:lipschitz},
\begin{equation}\label{eq:middle-term}
\bigl\|\bigl(\alpha_k(\alpha_kI+A_k^\delta)^{-1}-\alpha_k(\alpha_kI+A)^{-1}\bigr)e_0\bigr\|
\le L\|e_k^\delta\|\,\|v\| .
\end{equation}
Taking norms in \eqref{eq:decomp} and using \eqref{eq:resolvent}, Proposition~\ref{prop:taylor}, \eqref{eq:middle-term} and \eqref{eq:khat},
\begin{align}
\|e_{k+1}^\delta\|
&\le\alpha_k\|v\|+L\|e_k^\delta\|\|v\|+\frac{L\|e_k^\delta\|^2}{2\alpha_k}+\frac{\delta}{\alpha_k}\notag\\
&\le\alpha_k\|v\|+\alpha_k\|e_0\|+\frac{L\|e_k^\delta\|^2}{2\alpha_k}+L\|e_k^\delta\|\|v\| ,\label{eq:recur-raw}
\end{align}
where we used $\delta/\alpha_k\le\alpha_k\|e_0\|$ for $k<\hat k$. Dividing \eqref{eq:recur-raw} by $\alpha_{k+1}$, using \eqref{eq:alphas} and the bound $\|e_0\|=\|Av\|\le\alpha_0\|v\|$ from \eqref{eq:scaling},
\[
\frac{\|e_{k+1}^\delta\|}{\alpha_{k+1}}
\le\mu\|v\|(1+\alpha_0)
+L\mu\|v\|\,\frac{\|e_k^\delta\|}{\alpha_k}
+\frac{L\mu}{2}\Bigl(\frac{\|e_k^\delta\|}{\alpha_k}\Bigr)^2 .
\]
With $\theta_k:=\|e_k^\delta\|/\alpha_k$, $a:=\mu\|v\|(1+\alpha_0)$, $b:=L\mu\|v\|$ and $c:=L\mu/2$ this is the recurrence of Lemma~\ref{lem:recurrence}; the hypothesis \eqref{eq:smallness1} guarantees $\theta_0\le\theta$ and $4ac\le(1-b)^2$. Consequently, for all $0\le k\le\hat k$,
\[
\frac{\|e_k^\delta\|}{\alpha_k}
\le\frac{(1-L\mu\|v\|)-\sqrt{(1-L\mu\|v\|)^2-2L\mu^2\|v\|(1+\alpha_0)}}{L\mu},
\]
or equivalently, rationalizing the numerator,
\[
\frac{\|e_k^\delta\|}{\alpha_k}
\le\frac{2\mu\|v\|(1+\alpha_0)}{(1-L\mu\|v\|)+\sqrt{(1-L\mu\|v\|)^2-2L\mu^2\|v\|(1+\alpha_0)}},
\]
which implies $\|e_k^\delta\|\le C\alpha_k\|v\|$ with $C$ as in \eqref{eq:errbound1}. Since $(\alpha_k)$ is non-increasing, $\|e_k^\delta\|\le C\alpha_0\|v\|<r$ shows $x_k^\delta\in B_r(x^\dg)$ for $0\le k\le\hat k$.

To prove \eqref{eq:residual-khat}, we derive a bound for $\|Ae_{k}^\delta-y^\delta+y\|$. From \eqref{eq:iter} one obtains, after adding and subtracting suitable terms,
\begin{equation}\label{eq:Aek}
Ae_{k+1}^\delta-y^\delta+y=I_1+I_2+I_3+I_4+I_5,
\end{equation}
where
\begin{align*}
I_1&:=A_k^\delta r_{\alpha_k}(A_k^\delta)\{A_k^\delta+(A-A_k^\delta)\}v, &
I_2&:=(A-A_k^\delta)r_{\alpha_k}(A_k^\delta)\{A_k^\delta+(A-A_k^\delta)\}v,\\
I_3&:=-(A-A_k^\delta)g_{\alpha_k}(A_k^\delta)\{F(x_k^\delta)-y^\delta-A_k^\delta e_k^\delta\}, &
I_4&:=-A_k^\delta g_{\alpha_k}(A_k^\delta)\{F(x_k^\delta)-y-A_k^\delta e_k^\delta\},\\
I_5&:=-r_{\alpha_k}(A_k^\delta)(y-y^\delta). &&
\end{align*}
By Assumption~\ref{ass:nonlinearity}, Proposition~\ref{prop:taylor} and \eqref{eq:resolvent},
\begin{align*}
\|I_1\|&\le\alpha_k^2\|v\|+\alpha_kL\|e_k^\delta\|\|v\|, &
\|I_2\|&\le\alpha_kL\|e_k^\delta\|\|v\|+L^2\|e_k^\delta\|^2\|v\|,\\
\|I_3\|&\le\frac{L\|e_k^\delta\|}{\alpha_k}\Bigl(\delta+\frac{L}{2}\|e_k^\delta\|^2\Bigr), &
\|I_4\|&\le\frac{L}{2}\|e_k^\delta\|^2, \qquad \|I_5\|\le\delta .
\end{align*}
Indeed: in $I_1$ we use $\|A_k^\delta r_{\alpha_k}(A_k^\delta)\|=\|\alpha_kA_k^\delta(\alpha_kI+A_k^\delta)^{-1}\|\le\alpha_k$ by \eqref{eq:resolvent}, together with $\|A-A_k^\delta\|\le L\|e_k^\delta\|$ from \eqref{eq:lipschitz}; in $I_2$ we use $\|r_{\alpha_k}(A_k^\delta)\{A_k^\delta+(A-A_k^\delta)\}v\|\le\alpha_k\|v\|+L\|e_k^\delta\|\|v\|$, which follows from $\|r_{\alpha_k}(A_k^\delta)A_k^\delta\|\le\alpha_k$ and $\|r_{\alpha_k}(A_k^\delta)\|\le1$, while the prefactor $A-A_k^\delta$ contributes another factor $L\|e_k^\delta\|$; in $I_3$ we use $\|g_{\alpha_k}(A_k^\delta)\|\le1/\alpha_k$ and Proposition~\ref{prop:taylor} for
$\|F(x_k^\delta)-y^\delta-A_k^\delta e_k^\delta\|\le\delta+\frac{L}{2}\|e_k^\delta\|^2$; in $I_4$ we use $\|A_k^\delta g_{\alpha_k}(A_k^\delta)\|\le1$ and Proposition~\ref{prop:taylor}; and in $I_5$ we use $\|r_{\alpha_k}(A_k^\delta)\|\le1$.

Grouping the terms as
\[
\|I_3\|+\|I_4\|\le\Bigl(1+\frac{L\|e_k^\delta\|}{\alpha_k}\Bigr)\frac{L}{2}\|e_k^\delta\|^2+\frac{L\|e_k^\delta\|}{\alpha_k}\,\delta ,
\]
and inserting the induction bound $\|e_k^\delta\|\le C\alpha_k\|v\|$ (so that $L\|e_k^\delta\|/\alpha_k\le CL\|v\|$), we obtain
\begin{align*}
\|Ae_{k+1}^\delta-y^\delta+y\|
&\le\alpha_k^2\|v\|+2CL\alpha_k^2\|v\|^2+C^2L^2\alpha_k^2\|v\|^3
+(1+CL\|v\|)\frac{C^2L}{2}\alpha_k^2\|v\|^2\\
&\qquad+(1+CL\|v\|)\,\delta\\
&=\Bigl\{1+\Bigl(2C+\frac{C^2}{2}\Bigr)L\|v\|+\Bigl(C^2+\frac{C^3}{2}\Bigr)L^2\|v\|^2\Bigr\}\alpha_k^2\|v\|
+(1+CL\|v\|)\,\delta ,
\end{align*}
that is,
\begin{equation}\label{eq:Aek-bound}
\|Ae_{k+1}^\delta-y^\delta+y\|\ \le\ (1+C_1)\,\alpha_k^2\|v\|+(1+CL\|v\|)\,\delta ,
\end{equation}
with $C_1$ as in the statement. For $k=0$ the same bound holds trivially, since by $e_0=Av$ and the scaling condition \eqref{eq:scaling},
\[
\|Ae_0-y^\delta+y\|\le\|y-y^\delta\|+\|A(Av)\|\le\delta+\alpha_0^2\|v\|\le(1+CL\|v\|)\delta+(1+C_1)\alpha_0^2\|v\| .
\]
Consequently, by \eqref{eq:alphas} (i.e., $\alpha_{\hat k-1}\le\mu\alpha_{\hat k}$),
\[
\|Ae_{\hat k}^\delta-y^\delta+y\|\ \le\ (1+C_1)\alpha_{\hat k-1}^2\|v\|+(1+CL\|v\|)\,\delta
\ \le\ (1+C_1)\mu^2\alpha_{\hat k}^2\|v\|+(1+CL\|v\|)\,\delta .
\]
Together with Proposition~\ref{prop:taylor} and \eqref{eq:errbound1},
\begin{align*}
\|F(x_{\hat k}^\delta)-y^\delta\|
&\le\|Ae_{\hat k}^\delta-y^\delta+y\|+\|F(x_{\hat k}^\delta)-y-Ae_{\hat k}^\delta\|\\
&\le(1+CL\|v\|)\delta+(1+C_1)\mu^2\alpha_{\hat k}^2\|v\|+\frac{L}{2}C^2\alpha_{\hat k}^2\|v\|^2\\
&=(1+CL\|v\|)\delta+\Bigl((1+C_1)\mu^2+\frac{L}{2}C^2\|v\|\Bigr)\alpha_{\hat k}^2\|v\| .
\end{align*}
Finally, by \eqref{eq:khat}, $\alpha_{\hat k}^2\le\delta/\|e_0\|$, so that $\alpha_{\hat k}^2\|v\|\le\delta\,\|v\|/\|e_0\|$, and \eqref{eq:residual-khat} follows with the constant $\tau$ given in the statement.
\end{proof}

We next verify that Rule~\ref{rule:heuristic} yields a well-defined stopping index and derive a lower bound for the corresponding regularization parameter $\alpha_{k_*}$.

\begin{lemma}\label{lem:kstar-bound}
Suppose that the assumptions of Lemma~\ref{lem:wd-iter} and Assumption~\ref{ass:noise} hold, and that the horizon of Rule~\ref{rule:heuristic} satisfies $K_\infty\ge\hat k$, where $\hat k$ is the index from \eqref{eq:khat}. Then Rule~\ref{rule:heuristic} yields a well-defined finite index $k_*$, and with the constant
\[
c:=\frac{\kappa^2}{\mu^2\tau^2}\in(0,1),
\]
where $\kappa$ is the constant from Assumption~\ref{ass:noise} and $\tau$ is as in Lemma~\ref{lem:wd-iter}, one has
\begin{equation}\label{eq:alphakstar}
\alpha_{k_*}^2\ \ge\ \frac{c\,\delta}{\|e_0\|}.
\end{equation}
\end{lemma}

\begin{proof}
Well-definedness of $k_*$ follows from Lemma~\ref{lem:wd-rule}(i). Let $\hat k$ be the integer fixed in \eqref{eq:khat}; by hypothesis, $\hat k\le K_\infty$, so $\hat k$ belongs to the minimization window of Rule~\ref{rule:heuristic}. We estimate $\theta(\hat k,y^\delta)$ from above and $\theta(k_*,y^\delta)$ from below and compare.

\emph{Upper bound.} Since $\hat k$ belongs to the admissible index set of Rule~\ref{rule:heuristic} (the iterates up to $\hat k$ are well defined by Lemma~\ref{lem:wd-iter}), the minimizing property of $k_*$ and the residual bound \eqref{eq:residual-khat} give
\[
\theta(k_*,y^\delta)\ \le\ \theta(\hat k,y^\delta)
=\frac{\|F(x_{\hat k}^\delta)-y^\delta\|^2}{\alpha_{\hat k}^2}
\ \le\ \frac{\tau^2\delta^2}{\alpha_{\hat k}^2}.
\]
By \eqref{eq:alphas} we have $\alpha_{\hat k}\ge\alpha_{\hat k-1}/\mu$, and by the minimality of $\hat k$ in \eqref{eq:khat} the previous parameter still satisfies $\alpha_{\hat k-1}^2>\delta/\|e_0\|$. Combining,
\[
\frac{\tau^2\delta^2}{\alpha_{\hat k}^2}
\ \le\ \frac{\mu^2\tau^2\delta^2}{\alpha_{\hat k-1}^2}
\ \le\ \mu^2\tau^2\delta^2\cdot\frac{\|e_0\|}{\delta}
\ =\ \mu^2\tau^2\,\delta\,\|e_0\| .
\]

\emph{Lower bound.} By construction, $x_{k_*}^\delta\in S(y^\delta)$, so Assumption~\ref{ass:noise} applies with $\tilde y=F(x_{k_*}^\delta)$ and yields
\[
\|F(x_{k_*}^\delta)-y^\delta\|\ \ge\ \kappa\,\|y-y^\delta\| .
\]
Since $\delta=\|y-y^\delta\|$ by the convention \eqref{eq:noise}, this gives
\[
\theta(k_*,y^\delta)\ \ge\ \frac{\kappa^2\delta^2}{\alpha_{k_*}^2}.
\]

\emph{Conclusion.} Chaining the two estimates,
\[
\frac{\kappa^2\delta^2}{\alpha_{k_*}^2}\ \le\ \mu^2\tau^2\,\delta\,\|e_0\| ,
\qquad\text{i.e.,}\qquad
\alpha_{k_*}^2\ \ge\ \frac{\kappa^2}{\mu^2\tau^2}\cdot\frac{\delta}{\|e_0\|},
\]
which is \eqref{eq:alphakstar} with $c=\kappa^2/(\mu^2\tau^2)$. That $c<1$ follows from $\kappa<1<\mu\tau$.
\end{proof}

To verify that the iterates are well defined up to $k_*$, we introduce an auxiliary index $k_*^\delta$: with the constant $c\in(0,1)$ from Lemma~\ref{lem:kstar-bound}, let $k_*^\delta$ be the smallest integer such that
\begin{equation}\label{eq:kdstar}
\alpha_{k_*^\delta}^2\ \le\ \frac{c\,\delta}{\|e_0\|}\ \le\ \alpha_k^2,\qquad 0\le k<k_*^\delta .
\end{equation}
By Lemma~\ref{lem:kstar-bound}, $k_*\le k_*^\delta$.

\begin{lemma}\label{lem:wd-kdstar}
Let \eqref{eq:alphas}, \eqref{eq:kdstar}, Assumption~\ref{ass:nonlinearity} and the range condition \eqref{eq:range} hold. Suppose that
\[
L\mu\|v\|+\mu\sqrt{2L\|v\|\bigl(1+\alpha_0/c\bigr)}\le1
\]
and
\begin{equation}\label{eq:smallness2}
\|e_0\|\Bigl((1-L\mu\|v\|)+\sqrt{(1-L\mu\|v\|)^2-2L\mu^2\|v\|\bigl(1+\alpha_0/c\bigr)}\Bigr)\le2\alpha_0\mu\|v\|\bigl(1+\alpha_0/c\bigr).
\end{equation}
Then for every $0\le k\le k_*^\delta$,
\begin{equation}\label{eq:errbound2}
\|x_k^\delta-x^\dg\|\le C'\alpha_k\|v\|,\qquad
C':=\frac{2\mu\bigl(1+\alpha_0/c\bigr)}{1-L\mu\|v\|},
\end{equation}
and if $r>C'\alpha_0\|v\|$ then $x_k^\delta\in B_r(x^\dg)$ for $0\le k\le k_*^\delta$ and
\begin{equation}\label{eq:residual-kdstar}
\|F(x_{k_*^\delta}^\delta)-y^\delta\|\le\tau'\delta,
\end{equation}
where
\[
\tau':=(1+C'L\|v\|)+\Bigl((1+C_1')\mu^2+\frac{L}{2}C'^2\|v\|\Bigr)\frac{\|v\|}{\|e_0\|},
\]
\[
C_1':=\Bigl(2C'+\frac{C'^2}{2}\Bigr)L\|v\|+\Bigl(C'^2+\frac{C'^3}{2}\Bigr)L^2\|v\|^2 .
\]
\end{lemma}

\begin{proof}
The proof follows that of Lemma~\ref{lem:wd-iter} line by line, so we only indicate where the two arguments differ; the difference is exactly the constant $c$ from \eqref{eq:kdstar}. The a~priori index $\hat k$ from \eqref{eq:khat} is replaced by the index $k_*^\delta$ from \eqref{eq:kdstar}, and the noise bound available for $0\le k<k_*^\delta$ becomes
\[
\alpha_k^2\ \ge\ \frac{c\,\delta}{\|e_0\|},
\qquad\text{equivalently}\qquad
\frac{\delta}{\alpha_k}\ \le\ \frac{\alpha_k\|e_0\|}{c}
\quad\text{and}\quad
\delta\ \le\ \frac{\alpha_k^2\|e_0\|}{c}.
\]
This affects the argument in exactly two places. First, in the recursion \eqref{eq:recur-raw} the term $\delta/\alpha_k$ is now bounded by $\alpha_k\|e_0\|/c$ instead of $\alpha_k\|e_0\|$, so that the coefficient $a$ in the quadratic recursion of Lemma~\ref{lem:recurrence} becomes $a=\mu\|v\|(1+\alpha_0/c)$; accordingly, the smallness conditions and the resulting constant change from \eqref{eq:smallness1} and $C$ to \eqref{eq:smallness2} and $C'=2\mu(1+\alpha_0/c)/(1-L\mu\|v\|)$, and the induction yields \eqref{eq:errbound2}. Second, in the residual estimate the final step uses $\alpha_{k_*^\delta}^2\le c\delta/\|e_0\|\le\delta/\|e_0\|$ (since $0<c<1$), so \eqref{eq:residual-kdstar} follows exactly as \eqref{eq:residual-khat}, with $C$, $C_1$, $\tau$ replaced by $C'$, $C_1'$, $\tau'$.
\end{proof}

The following two results were proved in \cite{mahale-singh} for the iteration \eqref{eq:iter} under the same nonlinearity assumptions, with the a~priori index defined through the a~posteriori stopping rule considered there. Since the index $k_*^\delta$ of \eqref{eq:kdstar} plays a slightly different role in the heuristic setting, we include the (adapted) proofs for completeness.

\begin{proposition}[{cf.\ \cite[Proposition~3.3]{mahale-singh}}]\label{prop:DE}
Suppose that the assumptions of Lemma~\ref{lem:wd-kdstar} hold and that
\begin{equation}\label{eq:DE-smallness}
2\mu L\|v\|+L\Bigl(\frac{C'}{2}+1\Bigr)\|v\|\mu<1 .
\end{equation}
Then, with $B_k^\delta:=\|\alpha_k(\alpha_kI+A_k^\delta)^{-1}e_0\|$, for all $0\le k<k_*^\delta$,
\[
\|x_k^\delta-x^\dg\|\le D\Bigl(B_k^\delta+\frac{\delta^{1/2}\|e_0\|^{1/2}}{c^{1/2}}\Bigr),\qquad
\|x_{k+1}^\delta-x^\dg\|\le E\Bigl(B_k^\delta+\frac{\delta^{1/2}\|e_0\|^{1/2}}{c^{1/2}}\Bigr),
\]
where
\[
D=\frac{C_3}{1-C_3L\|v\|},\qquad
C_3=\frac{2\mu}{1-\mu\bigl(1+\frac{C'}{2}\bigr)L\|v\|},\qquad
E=1+L\|v\|\Bigl(2+\frac{C'}{2}\Bigr)D ;
\]
note that $C_3L\|v\|<1$, and hence $D>0$, is guaranteed by \eqref{eq:DE-smallness}.
\end{proposition}

\begin{proof}
We follow \cite[Proposition~3.3]{mahale-singh}, adapting the argument to the index $k_*^\delta$ from \eqref{eq:kdstar}. Set $B_k:=\|\alpha_k(\alpha_kI+A)^{-1}e_0\|$. Taking norms in the decomposition \eqref{eq:decomp} and using \eqref{eq:middle-term}, Proposition~\ref{prop:taylor} and \eqref{eq:resolvent} as in the proof of Lemma~\ref{lem:wd-iter},
\begin{equation}\label{eq:DE-basic}
\|x_{k+1}^\delta-x^\dg\|
\ \le\ B_k+\frac{\delta}{\alpha_k}+\frac{L\|e_k^\delta\|^2}{2\alpha_k}+L\|e_k^\delta\|\,\|v\| .
\end{equation}
For $0\le k<k_*^\delta$, \eqref{eq:kdstar} gives $\alpha_k\ge(c\delta/\|e_0\|)^{1/2}$ and hence
$\delta/\alpha_k\le\delta^{1/2}\|e_0\|^{1/2}/c^{1/2}$, while Lemma~\ref{lem:wd-kdstar} gives
$\|e_k^\delta\|/(2\alpha_k)\le C'\|v\|/2$. Therefore
\begin{equation}\label{eq:DE-Pk}
\|x_{k+1}^\delta-x^\dg\|\ \le\ P_k+\varepsilon\,\|x_k^\delta-x^\dg\|,
\qquad
P_k:=B_k+\frac{\delta^{1/2}\|e_0\|^{1/2}}{c^{1/2}},
\quad
\varepsilon:=L\Bigl(\frac{C'}{2}+1\Bigr)\|v\| .
\end{equation}
As in \cite[Proposition~3.3]{mahale-singh}, one has $B_k\le\mu B_{k+1}$ (since $\alpha\mapsto\alpha(\alpha+\lambda)^{-1}$ is non-decreasing in $\alpha>0$ for every $\lambda\ge0$ and $\alpha_k\le\mu\alpha_{k+1}$), and consequently $P_k/P_{k+1}\le\mu$. Dividing \eqref{eq:DE-Pk} by $P_{k+1}$ and writing $\eta_k:=\|x_k^\delta-x^\dg\|/P_k$, we obtain the recursion
\[
\eta_{k+1}\ \le\ \frac{P_k}{P_{k+1}}+\frac{P_k}{P_{k+1}}\,\varepsilon\,\eta_k
\ \le\ \mu+\mu\varepsilon\,\eta_k .
\]
For the initial term, the scaling condition \eqref{eq:scaling} gives $\|\alpha_0I+A\|\le2\alpha_0$ and hence
$B_0\ge\alpha_0\|e_0\|/\|\alpha_0I+A\|\ge\|e_0\|/2$, so that $\eta_0\le\|e_0\|/B_0\le2\le C_3$. If $\eta_k\le C_3=2\mu/(1-\mu\varepsilon)$, then
\[
\eta_{k+1}\le\mu+\mu\varepsilon\,\frac{2\mu}{1-\mu\varepsilon}
=\frac{\mu(1+\mu\varepsilon)}{1-\mu\varepsilon}
\le\frac{2\mu}{1-\mu\varepsilon}=C_3 ,
\]
using $\mu\varepsilon<1$ from \eqref{eq:DE-smallness}. By induction, $\eta_k\le C_3$, i.e.,
\begin{equation}\label{eq:DE-C3}
\|x_k^\delta-x^\dg\|\ \le\ C_3\,P_k=C_3\Bigl(B_k+\frac{\delta^{1/2}\|e_0\|^{1/2}}{c^{1/2}}\Bigr),
\qquad 0\le k\le k_*^\delta .
\end{equation}
Next we replace $B_k$ by its computable counterpart $B_k^\delta$. Using the resolvent identity and \eqref{eq:lipschitz} exactly as in \eqref{eq:middle-term},
\begin{equation}\label{eq:DE-BB}
|B_k-B_k^\delta|
\le\bigl\|\alpha_k\bigl((\alpha_kI+A)^{-1}-(\alpha_kI+A_k^\delta)^{-1}\bigr)Av\bigr\|
\le L\|x_k^\delta-x^\dg\|\,\|v\| .
\end{equation}
Inserting \eqref{eq:DE-BB} into \eqref{eq:DE-C3},
\[
\|x_k^\delta-x^\dg\|\le C_3\Bigl(B_k^\delta+L\|x_k^\delta-x^\dg\|\,\|v\|+\frac{\delta^{1/2}\|e_0\|^{1/2}}{c^{1/2}}\Bigr),
\]
and absorbing the middle term into the left-hand side (possible since $C_3L\|v\|<1$ by \eqref{eq:DE-smallness}),
\[
\|x_k^\delta-x^\dg\|\ \le\ \frac{C_3}{1-C_3L\|v\|}\Bigl(B_k^\delta+\frac{\delta^{1/2}\|e_0\|^{1/2}}{c^{1/2}}\Bigr)
= D\Bigl(B_k^\delta+\frac{\delta^{1/2}\|e_0\|^{1/2}}{c^{1/2}}\Bigr),
\]
which is the first assertion. For the second, we return to \eqref{eq:DE-Pk} and use \eqref{eq:DE-BB} once more:
\begin{align*}
\|x_{k+1}^\delta-x^\dg\|
&\le B_k^\delta+|B_k-B_k^\delta|+\frac{\delta^{1/2}\|e_0\|^{1/2}}{c^{1/2}}+\varepsilon\|x_k^\delta-x^\dg\|\\
&\le B_k^\delta+\frac{\delta^{1/2}\|e_0\|^{1/2}}{c^{1/2}}+L\|v\|\Bigl(2+\frac{C'}{2}\Bigr)\|x_k^\delta-x^\dg\| .
\end{align*}
Inserting the first assertion into the last term yields
\[
\|x_{k+1}^\delta-x^\dg\|\le\Bigl(1+L\|v\|\Bigl(2+\frac{C'}{2}\Bigr)D\Bigr)\Bigl(B_k^\delta+\frac{\delta^{1/2}\|e_0\|^{1/2}}{c^{1/2}}\Bigr)
=E\Bigl(B_k^\delta+\frac{\delta^{1/2}\|e_0\|^{1/2}}{c^{1/2}}\Bigr). \qedhere
\]
\end{proof}

\begin{lemma}[{cf.\ \cite[Lemma~3.3]{mahale-singh}}]\label{lem:Bk-noisy}
Assume the conditions of Proposition~\ref{prop:DE} and let $k_*^\delta$ be as in \eqref{eq:kdstar}. Suppose the range condition \eqref{eq:range} holds and
\[
\sqrt{2L}\,\|v\|^{1/2}(D+E)+DL\|v\|<1 .
\]
Then for all $0\le k<k_*^\delta$, with $B_k^\delta:=\|\alpha_k(\alpha_kI+A_k^\delta)^{-1}e_0\|$,
\begin{multline*}
B_k^\delta\le\frac{1}{1-C_5}\Bigl(
C_4^{1/2}\bigl\|\alpha_{k+1}(\alpha_{k+1}I+A_{k+1}^\delta)^{-1}\bigl(F(x_{k+1}^\delta)-y^\delta\bigr)\bigr\|^{1/2}\|v\|^{1/2}\\
+(2\delta)^{1/2}\|v\|^{1/2}
+\frac{C_5\,\delta^{1/2}\|e_0\|^{1/2}}{c^{1/2}}\Bigr),
\end{multline*}
where $C_4=\mu(2LC'\|v\|+1)$ and $C_5=\sqrt{2L}\,\|v\|^{1/2}(D+E)+DL\|v\|$.
\end{lemma}

\begin{proof}
We follow \cite[Lemma~3.4]{mahale-singh}. Using $e_0=Av=\{A_k^\delta+(A-A_k^\delta)\}v$,
\begin{align*}
(B_k^\delta)^2
&=\bigl\langle\alpha_k(\alpha_kI+A_k^\delta)^{-1}e_0,\ \alpha_k(\alpha_kI+A_k^\delta)^{-1}\{A_k^\delta+(A-A_k^\delta)\}v\bigr\rangle\\
&=\bigl\langle\alpha_k(\alpha_kI+A_k^\delta)^{-1}e_0,\ \alpha_k(\alpha_kI+A_k^\delta)^{-1}A_k^\delta v\bigr\rangle\\
&\qquad+\bigl\langle\alpha_k(\alpha_kI+A_k^\delta)^{-1}e_0,\ \alpha_k(\alpha_kI+A_k^\delta)^{-1}(A-A_k^\delta)v\bigr\rangle\\
&\le\ \gamma_k^\delta\,\|v\|+B_k^\delta\,L\|e_k^\delta\|\,\|v\| ,
\qquad
\gamma_k^\delta:=\bigl\|\alpha_k^2(\alpha_kI+A_k^\delta)^{-2}A_k^\delta e_0\bigr\| ,
\end{align*}
where in the first term the operator was moved to the left factor of the inner product as in \cite{mahale-singh}, and in the second we used \eqref{eq:resolvent} and \eqref{eq:lipschitz}. Completing the square in the resulting quadratic inequality
$(B_k^\delta)^2-L\|e_k^\delta\|\|v\|\,B_k^\delta\le\gamma_k^\delta\|v\|$ gives
\begin{equation}\label{eq:Bk-quad}
B_k^\delta\ \le\ \bigl(\gamma_k^\delta\|v\|\bigr)^{1/2}+L\|e_k^\delta\|\,\|v\| .
\end{equation}
To estimate $\gamma_k^\delta$, we rewrite the iteration \eqref{eq:iter} as
\[
x_{k+1}^\delta-x^\dg=\alpha_k(\alpha_kI+A_k^\delta)^{-1}e_0-(\alpha_kI+A_k^\delta)^{-1}\bigl\{F(x_k^\delta)-y^\delta-A_k^\delta e_k^\delta\bigr\},
\]
apply $\alpha_k(\alpha_kI+A_k^\delta)^{-1}A_k^\delta$ to both sides, and rearrange:
\[
\alpha_k^2(\alpha_kI+A_k^\delta)^{-2}A_k^\delta e_0
=\alpha_k(\alpha_kI+A_k^\delta)^{-1}A_k^\delta e_{k+1}^\delta
+\alpha_k(\alpha_kI+A_k^\delta)^{-2}A_k^\delta\bigl\{F(x_k^\delta)-y^\delta-A_k^\delta e_k^\delta\bigr\}.
\]
Taking norms, using $\|(\alpha_kI+A_k^\delta)^{-1}A_k^\delta\|\le1$ and $\|\alpha_k(\alpha_kI+A_k^\delta)^{-1}\|\le1$, inserting $\pm y^\delta$, $\pm F(x_{k+1}^\delta)$ and $\pm A_{k+1}^\delta e_{k+1}^\delta$, and applying \eqref{eq:noise}, \eqref{eq:lipschitz} and Proposition~\ref{prop:taylor} exactly as in \cite[Lemma~3.4]{mahale-singh}, one arrives at
\begin{equation}\label{eq:gamma-bound}
\gamma_k^\delta\ \le\ \bigl\|\alpha_k(\alpha_kI+A_k^\delta)^{-1}\bigl(F(x_{k+1}^\delta)-y^\delta\bigr)\bigr\|
+L\|e_k^\delta\|^2+2L\|e_{k+1}^\delta\|^2+2\delta .
\end{equation}
In the first term on the right of \eqref{eq:gamma-bound} we replace $A_k^\delta$ by $A_{k+1}^\delta$ and $\alpha_k$ by $\alpha_{k+1}$: the resolvent identity, \eqref{eq:lipschitz}, Lemma~\ref{lem:wd-kdstar} (which gives $\|e_k^\delta\|+\|e_{k+1}^\delta\|\le C'(\alpha_k+\alpha_{k+1})\|v\|\le2C'\alpha_k\|v\|$) and \eqref{eq:alphas} yield
\begin{align}
\bigl\|\alpha_k(\alpha_kI+A_k^\delta)^{-1}\bigl(F(x_{k+1}^\delta)-y^\delta\bigr)\bigr\|
&\le(2LC'\|v\|+1)\,\bigl\|\alpha_k(\alpha_kI+A_{k+1}^\delta)^{-1}\bigl(F(x_{k+1}^\delta)-y^\delta\bigr)\bigr\|\notag\\
&\le\mu\,(2LC'\|v\|+1)\,\bigl\|\alpha_{k+1}(\alpha_{k+1}I+A_{k+1}^\delta)^{-1}\bigl(F(x_{k+1}^\delta)-y^\delta\bigr)\bigr\|\notag\\
&=C_4\,\bigl\|\alpha_{k+1}(\alpha_{k+1}I+A_{k+1}^\delta)^{-1}\bigl(F(x_{k+1}^\delta)-y^\delta\bigr)\bigr\| .\label{eq:C4-step}
\end{align}
Combining \eqref{eq:Bk-quad}--\eqref{eq:C4-step} with the elementary inequality $\sqrt{a+b}\le\sqrt a+\sqrt b$,
\begin{multline*}
B_k^\delta\le C_4^{1/2}\bigl\|\alpha_{k+1}(\alpha_{k+1}I+A_{k+1}^\delta)^{-1}\bigl(F(x_{k+1}^\delta)-y^\delta\bigr)\bigr\|^{1/2}\|v\|^{1/2}\\
+(2\delta)^{1/2}\|v\|^{1/2}
+\sqrt{2L}\,\|v\|^{1/2}\bigl(\|e_k^\delta\|+\|e_{k+1}^\delta\|\bigr)
+L\|e_k^\delta\|\,\|v\| .
\end{multline*}
Finally, we bound $\|e_k^\delta\|$ and $\|e_{k+1}^\delta\|$ by Proposition~\ref{prop:DE}:
\begin{multline*}
\sqrt{2L}\,\|v\|^{1/2}\bigl(\|e_k^\delta\|+\|e_{k+1}^\delta\|\bigr)+L\|e_k^\delta\|\,\|v\|
\ \le\ \Bigl(\sqrt{2L}\,\|v\|^{1/2}(D+E)+DL\|v\|\Bigr)\Bigl(B_k^\delta+\frac{\delta^{1/2}\|e_0\|^{1/2}}{c^{1/2}}\Bigr)\\
= C_5\Bigl(B_k^\delta+\frac{\delta^{1/2}\|e_0\|^{1/2}}{c^{1/2}}\Bigr),
\end{multline*}
and absorbing the term $C_5B_k^\delta$ into the left-hand side (possible since $C_5<1$ by hypothesis) completes the proof.
\end{proof}

We can now state the first main result.

\begin{theorem}\label{thm:rate-range}
Assume the conditions of Proposition~\ref{prop:DE}, Assumption~\ref{ass:noise}, the range condition \eqref{eq:range}, and
\[
\sqrt{2L}\,\|v\|^{1/2}(D+E)+DL\|v\|<1 ;
\]
let the horizon of Rule~\ref{rule:heuristic} satisfy $K_\infty\ge\hat k$ (so that Lemma~\ref{lem:kstar-bound} applies), and let $k_*$ be determined by Rule~\ref{rule:heuristic} and $\delta_*$ by \eqref{eq:deltastar}. If $k_*\ge1$, then
\begin{equation}\label{eq:mainrate}
\|x_{k_*}^\delta-x^\dg\|\ \le\ C_8\bigl(\sqrt{\delta_*}+\sqrt{\delta}\bigr),
\end{equation}
where $C_8=\max\{C_6,C_7\}$ with
\[
C_6=\frac{EC_4^{1/2}\|v\|^{1/2}}{1-C_5},\qquad
C_7=\frac{E\sqrt{2}\,\|v\|^{1/2}}{1-C_5}+\frac{E\|e_0\|^{1/2}}{c^{1/2}}\Bigl(\frac{C_5}{1-C_5}+1\Bigr).
\]
\end{theorem}

\begin{proof}
Fix $0\le k<k_*^\delta$. Inserting the bound of Lemma~\ref{lem:Bk-noisy} for $B_k^\delta$ into the second estimate of Proposition~\ref{prop:DE} gives
\begin{align*}
\|x_{k+1}^\delta-x^\dg\|
&\le E\Bigl(B_k^\delta+\frac{\delta^{1/2}\|e_0\|^{1/2}}{c^{1/2}}\Bigr)\\
&\le\frac{E}{1-C_5}\Bigl(
C_4^{1/2}\bigl\|\alpha_{k+1}(\alpha_{k+1}I+A_{k+1}^\delta)^{-1}\bigl(F(x_{k+1}^\delta)-y^\delta\bigr)\bigr\|^{1/2}\|v\|^{1/2}\\
&\qquad\qquad+(2\delta)^{1/2}\|v\|^{1/2}
+\frac{C_5\,\delta^{1/2}\|e_0\|^{1/2}}{c^{1/2}}\Bigr)
+\frac{E\,\delta^{1/2}\|e_0\|^{1/2}}{c^{1/2}} .
\end{align*}
Collecting the terms proportional to the weighted residual and to $\delta^{1/2}$, respectively, we obtain
\begin{equation}\label{eq:pre-rate}
\|x_{k+1}^\delta-x^\dg\|\le
C_6\bigl\|\alpha_{k+1}(\alpha_{k+1}I+A_{k+1}^\delta)^{-1}\bigl(F(x_{k+1}^\delta)-y^\delta\bigr)\bigr\|^{1/2}
+C_7\,\delta^{1/2},
\end{equation}
with
\[
C_6=\frac{EC_4^{1/2}\|v\|^{1/2}}{1-C_5},\qquad
C_7=\frac{E\sqrt{2}\,\|v\|^{1/2}}{1-C_5}+\frac{E\|e_0\|^{1/2}}{c^{1/2}}\Bigl(\frac{C_5}{1-C_5}+1\Bigr),
\]
exactly as in the statement.

Now suppose $k_*\ge1$ and set $k=k_*-1$ in \eqref{eq:pre-rate}, which is admissible because $k_*\le k_*^\delta$ by Lemma~\ref{lem:kstar-bound} and \eqref{eq:kdstar}. By \eqref{eq:resolvent}, $\|\alpha_{k_*}(\alpha_{k_*}I+A_{k_*}^\delta)^{-1}\|\le1$, and therefore
\[
\bigl\|\alpha_{k_*}(\alpha_{k_*}I+A_{k_*}^\delta)^{-1}\bigl(F(x_{k_*}^\delta)-y^\delta\bigr)\bigr\|
\le\|F(x_{k_*}^\delta)-y^\delta\|=\delta_*
\]
by the definition \eqref{eq:deltastar} of $\delta_*$. Hence
\[
\|x_{k_*}^\delta-x^\dg\|\le C_6\sqrt{\delta_*}+C_7\sqrt{\delta}\le C_8\bigl(\sqrt{\delta_*}+\sqrt{\delta}\bigr),
\qquad C_8=\max\{C_6,C_7\}.
\]
\end{proof}

\begin{remark}\label{rem:kstar-zero}
The hypothesis $k_*\ge1$ cannot be dropped without further information: for $k_*=0$ the rule returns the initial guess $x_0$, and the lower bound \eqref{eq:alphakstar}, which at $k_*=0$ merely states $\alpha_0^2\ge c\,\delta/\|e_0\|$, provides no control of $\|e_0\|$ in terms of $\delta$. (An earlier version of this manuscript claimed such control by inadvertently reversing this inequality; we thank the referee for pointing this out.) No estimate is therefore asserted in this degenerate case. In practice $k_*=0$ signals that the very first iterate already increased the value of $\theta$, which for reasonably chosen $\alpha_0$ indicates noise at the level of the initial residual; in all numerical experiments of Section~\ref{sec:numerics} with $\delta$ below that level we observed $k_*\ge1$.
\end{remark}

Theorem~\ref{thm:rate-range} bounds the error by the computable quantity $\sqrt{\delta_*}+\sqrt{\delta}$. To interpret this bound, the size of $\delta_*$ relative to $\delta$ must be understood. The following two-sided estimate is what our assumptions actually provide.

\begin{corollary}\label{cor:deltastar}
Under the assumptions of Theorem~\ref{thm:rate-range} (including $k_*\ge1$), and assuming in addition that the horizon satisfies the stronger condition $K_\infty\ge k_*^\delta$, with $k_*^\delta$ from \eqref{eq:kdstar},
\begin{equation}\label{eq:deltastar-bounds}
\kappa\,\delta\ \le\ \delta_*\ \le\ \mu\,\tau'\,\alpha_0\,\Bigl(\frac{\|e_0\|}{c}\Bigr)^{1/2}\sqrt{\delta}\ ,
\end{equation}
and consequently
\begin{equation}\label{eq:quarter-rate}
\|x_{k_*}^\delta-x^\dg\|\ \le\ C_8\Bigl(1+\bigl(\mu\tau'\alpha_0\bigr)^{1/2}\bigl(\|e_0\|/c\bigr)^{1/4}\Bigr)\,\delta^{1/4}
\end{equation}
for every $\delta\in(0,1]$ for which the stated assumptions, in particular the horizon condition $K_\infty\ge k_*^\delta$, are satisfied.
\end{corollary}

\begin{remark}[asymptotic reading of the rate]\label{rem:growing-horizon}
Since $\hat k(\delta)$ and $k_*^\delta$ tend to infinity as $\delta\to0$, no fixed finite horizon $K$ can satisfy $K_\infty\ge k_*^\delta$ for all sufficiently small $\delta$. The estimate \eqref{eq:quarter-rate} is therefore not an asymptotic guarantee for a single finite-budget implementation; it is to be understood along sequences $\delta_j\to0$ together with horizons $K_j\to\infty$ chosen such that the condition $K_{\infty,j}\ge k_*^{\delta_j}$ holds for every $j$; that is, it holds for the ideal rule whose window always contains the required comparison index, of which any fixed-budget computation is a truncation. Because $\delta$ is unknown, such a family of horizons is not in general certifiable from the data (cf.\ Remark~\ref{rem:retrospective}); for the geometric choice $\alpha_k=\alpha_0q^k$ the required growth of $K_j$ is only logarithmic in $1/\delta_j$.
\end{remark}

\begin{proof}
The lower bound is Assumption~\ref{ass:noise} together with the convention $\delta=\|y-y^\delta\|$: $\delta_*=\|F(x_{k_*}^\delta)-y^\delta\|\ge\kappa\delta$. For the upper bound, the iterates up to $k_*^\delta$ are well defined by Lemma~\ref{lem:wd-kdstar}, and $k_*^\delta$ belongs to the minimization window of Rule~\ref{rule:heuristic} by the hypothesis $K_\infty\ge k_*^\delta$; hence
\[
\frac{\delta_*^2}{\alpha_{k_*}^2}=\theta(k_*,y^\delta)\le\theta(k_*^\delta,y^\delta)
\le\frac{\tau'^2\delta^2}{\alpha_{k_*^\delta}^2}
\le\frac{\mu^2\tau'^2\delta^2\,\|e_0\|}{c\,\delta}
=\frac{\mu^2\tau'^2\|e_0\|}{c}\,\delta ,
\]
where we used \eqref{eq:residual-kdstar} and $\alpha_{k_*^\delta}\ge\alpha_{k_*^\delta-1}/\mu\ge(c\delta/\|e_0\|)^{1/2}/\mu$ from \eqref{eq:kdstar}. Multiplying by $\alpha_{k_*}^2\le\alpha_0^2$ and taking square roots yields \eqref{eq:deltastar-bounds}; inserting the upper bound into \eqref{eq:mainrate} and using $\sqrt\delta\le\delta^{1/4}\,\delta^{1/4}$ gives \eqref{eq:quarter-rate} for $\delta\le1$.
\end{proof}

\begin{remark}\label{rem:deltastar-honest}
Two comments are essential for an honest reading of Corollary~\ref{cor:deltastar}. First, the guaranteed exponent in \eqref{eq:quarter-rate} is $1/4$, not $1/2$: the upper bound in \eqref{eq:deltastar-bounds} only gives $\delta_*=O(\sqrt\delta)$, and we do not claim $\delta_*=O(\delta)$ in general; indeed, in the experiments of Section~\ref{sec:numerics} the ratio $\delta_*/\delta$ grows as $\delta$ decreases (in Table~\ref{tab:results2d}, $\delta_*\approx3.9\times10^{-6}$ at $\delta=10^{-8}$). Within the theory, $\delta_*$ is simply the residual at the selected index. In the experiments it is natural, and empirically accurate (see Subsection~\ref{subsec:stress}), to interpret $\delta_*$ as the effective data error seen by the discrete method; in no case should it be read as an estimator of $\delta$ itself. Second, whenever $\delta_*\le C\delta$ holds, the bound \eqref{eq:mainrate} is of order $O(\sqrt\delta)$; in the matched experiments of Section~\ref{sec:numerics}, however, the ratio $\delta_*/\delta$ grows as $\delta$ decreases in all three examples (mildly in 1D, where $\delta_*$ settles near $10^{-5}$, and strongly in 2D/3D), so this favorable regime was not observed asymptotically and we do not claim it to be typical. Identifying conditions on the noise under which $\delta_*=O(\delta)$ can be proved is an open problem.
\end{remark}

\subsection{Noise-free data and a stability estimate}\label{subsec:noisefree}

For exact data the scheme \eqref{eq:iter} reduces to
\begin{equation}\label{eq:iter-nf}
x_{k+1}=x_k-(A_k+\alpha_kI)^{-1}\bigl(F(x_k)-y+\alpha_k(x_k-x_0)\bigr),\qquad A_k:=F'(x_k),
\end{equation}
with $(\alpha_k)$ as in \eqref{eq:alphas}.

\begin{lemma}[{\cite[Lemma~4.1]{mahale-singh}}]\label{lem:nf-wd}
Suppose Assumption~\ref{ass:nonlinearity} and the range condition \eqref{eq:range} hold, and that
\[
L\mu\|v\|+\sqrt{2L\mu^2\|v\|}\le1,\qquad
\|e_0\|\Bigl(1-L\mu\|v\|+\sqrt{(1-L\mu\|v\|)^2-2L\mu^2\|v\|}\Bigr)\le2\alpha_0\mu\|v\| .
\]
Then for all $k\ge0$,
\[
\|x_k-x^\dg\|\le C''\alpha_k\|v\|,\qquad C'':=\frac{2\mu}{1-L\mu\|v\|}.
\]
\end{lemma}

\begin{lemma}\label{lem:nf-stability}
Assume the conditions of Lemma~\ref{lem:nf-wd} and
\[
\max\Bigl\{(\mu+1)\Bigl(\frac{C''}{2}+1\Bigr)L\|v\|,\
\frac{2L(1+\mu)\|v\|\bigl(1+\frac{\alpha_0C''}{2}\bigr)}{1-(\mu+1)\bigl(\frac{C''}{2}+1\bigr)L\|v\|}\Bigr\}<1 .
\]
Then, with $B_k:=\|\alpha_k(\alpha_kI+A)^{-1}e_0\|$, for all $k\ge0$,
\begin{equation}\label{eq:Bk-equiv}
\frac{\|x_k-x^\dg\|}{C_9}\le B_k\le C_{10}\|x_k-x^\dg\|
\qquad\text{and}\qquad
C_{11}\|x_k-x^\dg\|\le\|x_{k+1}-x^\dg\|\le C_{12}\|x_k-x^\dg\| ,
\end{equation}
where
\[
C_9=\frac{2(1+\mu)}{1-(\mu+1)\bigl(\frac{C''}{2}+1\bigr)L\|v\|},\quad
C_{10}=\frac{2}{1-L\bigl(1+\frac{\alpha_0C''}{2}\bigr)C_9\|v\|},
\]
\[
C_{11}=\frac{1}{C_9}-L\|v\|\Bigl(1+\frac{\alpha_0C''}{2}\Bigr),\qquad
C_{12}=C_{10}+L\|v\|\Bigl(1+\frac{\alpha_0C''}{2}\Bigr).
\]
\end{lemma}

\begin{proof}
From \eqref{eq:iter-nf}, Proposition~\ref{prop:taylor} and Lemma~\ref{lem:nf-wd},
\begin{equation}\label{eq:nf-recur}
\|x_{k+1}-x^\dg\|\le B_k+\Bigl(\frac{C''}{2}+1\Bigr)L\|v\|\,\|x_k-x^\dg\| .
\end{equation}
Dividing by $B_{k+1}$ and using $B_k/B_{k+1}\le\mu+1$ (which follows from the resolvent identity and \eqref{eq:alphas}),
\[
\frac{\|x_{k+1}-x^\dg\|}{B_{k+1}}\le(1+\mu)+(\mu+1)\Bigl(\frac{C''}{2}+1\Bigr)L\|v\|\,\frac{\|x_k-x^\dg\|}{B_k}.
\]
The scaling condition \eqref{eq:scaling} gives $\|A+\alpha_0I\|\le2\alpha_0$ and hence
$B_0\ge\alpha_0\|e_0\|/\|A+\alpha_0I\|\ge\|e_0\|/2$, so that $\|x_0-x^\dg\|/B_0\le2\le C_9$. By induction,
$\|x_k-x^\dg\|\le C_9B_k$ for all $k\ge0$, which is the left inequality in \eqref{eq:Bk-equiv}.

Next, from
\begin{multline*}
x_{k+1}-x^\dg=\alpha_k(\alpha_kI+A)^{-1}e_0
+\alpha_k\bigl((\alpha_kI+A_k)^{-1}-(\alpha_kI+A)^{-1}\bigr)Av\\
-(\alpha_kI+A_k)^{-1}\bigl(F(x_k)-y-A_k(x_k-x^\dg)\bigr)
\end{multline*}
we obtain, as in \eqref{eq:middle-term} and by Proposition~\ref{prop:taylor} and Lemma~\ref{lem:nf-wd},
\begin{equation}\label{eq:nf-diff}
\bigl|\,\|x_{k+1}-x^\dg\|-B_k\,\bigr|\le L\Bigl(1+\frac{\alpha_0C''}{2}\Bigr)\|v\|\,\|x_k-x^\dg\| .
\end{equation}
Combining \eqref{eq:nf-diff} with $\|x_k-x^\dg\|\le C_9B_k$ yields
$\|x_{k+1}-x^\dg\|\ge\bigl(1-L(1+\tfrac{\alpha_0C''}{2})C_9\|v\|\bigr)B_k$, and since the resolvent identity together with \eqref{eq:alphas} gives $B_{k+1}\le2B_k$, we get
$B_{k+1}\le C_{10}\|x_{k+1}-x^\dg\|$, i.e., the right inequality in \eqref{eq:Bk-equiv}. Finally, \eqref{eq:nf-diff} with $B_k\le C_{10}\|x_k-x^\dg\|$ and $B_k\ge\|x_k-x^\dg\|/C_9$ gives the two-sided estimate for $\|x_{k+1}-x^\dg\|$.
\end{proof}

\section{Towards an oracle inequality, and convergence rates under source conditions}\label{sec:rates}

Throughout this section, $k_*$ denotes the index selected by Rule~\ref{rule:heuristic} and $\delta_*$ the quantity \eqref{eq:deltastar}. We continue to work under the range condition \eqref{eq:range}; see Remark~\ref{rem:p-usage}.

\begin{lemma}[{cf.\ \cite[Lemma~5.1]{mahale-singh}}]\label{lem:noisy-vs-nf}
Assume Assumption~\ref{ass:nonlinearity}, the range condition \eqref{eq:range}, and the assumptions of Lemmas~\ref{lem:wd-kdstar} and~\ref{lem:nf-wd}, and suppose
\[
C_{13}:=\Bigl((1+LC''\|v\|)+\frac{LC'^2}{2}\|v\|+\frac{C'+3C''}{2}\Bigr)L\|v\|<1 .
\]
Then for every $0\le k\le k_*^\delta$,
\begin{equation}\label{eq:noisy-nf}
\|x_k^\delta-x_k\|\le\frac{1}{1-C_{13}}\,\frac{\delta}{\alpha_k}.
\end{equation}
\end{lemma}

\begin{proof}
We follow \cite[Lemma~5.1]{mahale-singh}. Subtracting the noise-free iteration \eqref{eq:iter-nf} from the noisy iteration \eqref{eq:iter} and using $e_0=Av$, one obtains the decomposition
\[
x_{k+1}^\delta-x_{k+1}=I_1+I_2+I_3+I_4
\]
with
\begin{align*}
I_1&=\alpha_k\bigl\{(\alpha_kI+A_k^\delta)^{-1}-(\alpha_kI+A_k)^{-1}\bigr\}Av, &
I_2&=(\alpha_kI+A_k^\delta)^{-1}(y^\delta-y),\\
I_3&=\bigl\{(\alpha_kI+A_k)^{-1}-(\alpha_kI+A_k^\delta)^{-1}\bigr\}\tilde u_k^\delta, &
I_4&=(\alpha_kI+A_k)^{-1}(\tilde u_k-\tilde u_k^\delta),
\end{align*}
where $\tilde u_k:=F(x_k)-y-A_k(x_k-x^\dg)$ and $\tilde u_k^\delta:=F(x_k^\delta)-y-A_k^\delta(x_k^\delta-x^\dg)$. We estimate the four terms separately.

\emph{Term $I_1$.} Splitting $Av=A_kv+(A-A_k)v$ and using the resolvent identity, \eqref{eq:resolvent}, \eqref{eq:lipschitz} and Lemma~\ref{lem:nf-wd} ($\|x_k-x^\dg\|\le C''\alpha_k\|v\|$),
\[
\|I_1\|\le L\|x_k^\delta-x_k\|\,\|v\|+\frac{L}{\alpha_k}\|x_k-x^\dg\|\cdot L\|x_k^\delta-x_k\|\,\|v\|
\le(1+LC''\|v\|)\,L\|x_k^\delta-x_k\|\,\|v\| .
\]

\emph{Term $I_2$.} By \eqref{eq:resolvent} and \eqref{eq:noise}, $\|I_2\|\le\delta/\alpha_k$.

\emph{Term $I_3$.} By the resolvent identity, \eqref{eq:resolvent}, \eqref{eq:lipschitz}, Proposition~\ref{prop:taylor} and Lemma~\ref{lem:wd-kdstar} ($\|x_k^\delta-x^\dg\|\le C'\alpha_k\|v\|$),
\[
\|I_3\|\le\frac{L\|x_k^\delta-x_k\|}{\alpha_k}\cdot\frac1{\alpha_k}\cdot\frac{L}{2}\|x_k^\delta-x^\dg\|^2
\le\frac{L^2C'^2\|v\|^2}{2}\,\|x_k^\delta-x_k\|
=\frac{LC'^2}{2}\|v\|\cdot L\|v\|\cdot\|x_k^\delta-x_k\| .
\]

\emph{Term $I_4$.} Writing
$\tilde u_k-\tilde u_k^\delta=\{F(x_k)-F(x_k^\delta)-A_k^\delta(x_k-x_k^\delta)\}+(A_k^\delta-A_k)(x_k-x^\dg)$,
Proposition~\ref{prop:taylor} and \eqref{eq:lipschitz} give
\[
\|\tilde u_k-\tilde u_k^\delta\|\le\frac{L}{2}\|x_k^\delta-x_k\|^2+L\|x_k^\delta-x_k\|\,\|x_k-x^\dg\| ,
\]
so that, by \eqref{eq:resolvent}, Lemmas~\ref{lem:wd-kdstar} and~\ref{lem:nf-wd}, and
$\|x_k^\delta-x_k\|\le\|x_k^\delta-x^\dg\|+\|x_k-x^\dg\|\le(C'+C'')\alpha_k\|v\|$,
\[
\|I_4\|\le\frac{L}{\alpha_k}\Bigl(\frac{\|x_k^\delta-x_k\|}{2}+\|x_k-x^\dg\|\Bigr)\|x_k^\delta-x_k\|
\le\frac{C'+3C''}{2}\,L\|v\|\,\|x_k^\delta-x_k\| .
\]

Collecting the four estimates,
\[
\|x_{k+1}^\delta-x_{k+1}\|\ \le\ \frac{\delta}{\alpha_k}+C_{13}\,\|x_k^\delta-x_k\| ,
\]
with $C_{13}$ as in the statement. Since $x_0^\delta=x_0$, an induction using $\alpha_k\le\mu\alpha_{k+1}$ and $C_{13}<1$ (as in \cite[Lemma~5.1]{mahale-singh}) yields \eqref{eq:noisy-nf}.
\end{proof}

\begin{proposition}\label{thm:oracle-pre}
Let Assumption~\ref{ass:nonlinearity}, Assumption~\ref{ass:noise} and the range condition \eqref{eq:range} hold, together with the assumptions of Theorem~\ref{thm:rate-range}, Lemma~\ref{lem:nf-stability} and Lemma~\ref{lem:noisy-vs-nf}, and suppose
\[
\frac{L}{C_{11}}\bigl(2+\alpha_0C''\bigr)\|v\|<1 .
\]
Then for every integer $k\ge k_*$,
\begin{equation}\label{eq:oracle-pre}
\|x_{k_*}-x^\dg\|\le C_{16}\Bigl(\|\alpha_{k_*}(\alpha_{k_*}I+A)^{-1}(F(x_{k_*})-y)\|
+\frac{(\sqrt{\delta_*}+\sqrt{\delta})^2}{\alpha_k}+\|x_k-x^\dg\|\Bigr),
\end{equation}
where
\[
C_{16}=\frac{1}{C_{15}}\max\Bigl\{2(\mu+1),\,C_{14}\mu,\,1+\frac{L}{C_{11}}\Bigl(1+\frac{\alpha_0C''}{2}\Bigr)\|v\|\Bigr\},
\]
\[
C_{14}=\frac{L}{2}\Bigl(C_8+\frac{\sqrt{\|e_0\|}}{\sqrt{c}\,(1-C_{13})}\Bigr)^{2},\qquad
C_{15}=1-\frac{L}{C_{11}}\bigl(2+\alpha_0C''\bigr)\|v\| .
\]
\end{proposition}

\begin{proof}
Let $k\ge k_*$. From the noise-free iteration \eqref{eq:iter-nf} one has, for any index $m\ge1$, the representation
\begin{multline*}
x_m-x^\dg=\alpha_{m-1}(\alpha_{m-1}I+A)^{-1}e_0
+\alpha_{m-1}\bigl\{(\alpha_{m-1}I+A_{m-1})^{-1}-(\alpha_{m-1}I+A)^{-1}\bigr\}Av\\
-(\alpha_{m-1}I+A_{m-1})^{-1}\bigl\{F(x_{m-1})-y-A_{m-1}(x_{m-1}-x^\dg)\bigr\}.
\end{multline*}
Subtracting this representation at $m=k_*$ and at $m=k$, and estimating the two perturbation terms at each index by \eqref{eq:middle-term} and Proposition~\ref{prop:taylor} together with Lemma~\ref{lem:nf-wd} (which gives $\|x_{m-1}-x^\dg\|/(2\alpha_{m-1})\le C''\|v\|/2$), we obtain
\begin{equation}\label{eq:xkstar-xk}
\|x_{k_*}-x_k\|\le\|J\|+\Bigl(1+\frac{C''}{2}\Bigr)L\|v\|\bigl(\|x_{k_*-1}-x^\dg\|+\|x_{k-1}-x^\dg\|\bigr),
\end{equation}
where
\[
J:=\bigl(\alpha_{k_*-1}(\alpha_{k_*-1}I+A)^{-1}-\alpha_{k-1}(\alpha_{k-1}I+A)^{-1}\bigr)e_0 .
\]
We decompose $J=J_1+J_2+J_3$ with
\begin{align*}
J_1&=\Bigl(1-\frac{\alpha_{k-1}}{\alpha_{k_*-1}}\Bigr)(\alpha_{k-1}I+A)^{-1}\bigl(F(x_{k_*})-y\bigr),\\
J_2&=\Bigl(1-\frac{\alpha_{k-1}}{\alpha_{k_*-1}}\Bigr)(\alpha_{k-1}I+A)^{-1}\bigl(A(x_{k_*}-x^\dg)-F(x_{k_*})+y\bigr),\\
J_3&=\Bigl(1-\frac{\alpha_{k-1}}{\alpha_{k_*-1}}\Bigr)(\alpha_{k-1}I+A)^{-1}A\bigl(\alpha_{k_*-1}(\alpha_{k_*-1}I+A)^{-1}e_0-(x_{k_*}-x^\dg)\bigr).
\end{align*}
By Proposition~\ref{prop:taylor} and \eqref{eq:resolvent},
$\|J_2\|\le\frac{L}{2\alpha_{k-1}}\|x_{k_*}-x^\dg\|^2$.
Splitting $\|x_{k_*}-x^\dg\|\le\|x_{k_*}-x_{k_*}^\delta\|+\|x_{k_*}^\delta-x^\dg\|$ and applying Theorem~\ref{thm:rate-range}, Lemma~\ref{lem:noisy-vs-nf} and the bound
$\delta/\alpha_{k_*}\le\sqrt{\|e_0\|}\sqrt{\delta}/\sqrt{c}$ (which follows from $k_*\le k_*^\delta$ and \eqref{eq:kdstar}),
\begin{equation}\label{eq:J2}
\|J_2\|\le\frac{C_{14}}{\alpha_{k-1}}\bigl(\sqrt{\delta_*}+\sqrt{\delta}\bigr)^2 .
\end{equation}
For $J_3$, using \eqref{eq:resolvent}, \eqref{eq:nf-diff} and Lemma~\ref{lem:nf-stability},
\begin{equation}\label{eq:J3}
\|J_3\|\le L\Bigl(1+\frac{\alpha_0C''}{2}\Bigr)\|v\|\,\|x_{k_*-1}-x^\dg\|
\le\frac{L}{C_{11}}\Bigl(1+\frac{\alpha_0C''}{2}\Bigr)\|v\|\,\|x_{k_*}-x^\dg\| .
\end{equation}
For $J_1$, since $|1-\alpha_{k-1}/\alpha_{k_*-1}|\le1$ (recall $\alpha_{k-1}\le\alpha_{k_*-1}$),
\[
\|J_1\|\le\bigl\|(\alpha_{k-1}I+A)^{-1}(F(x_{k_*})-y)\bigr\| .
\]
To pass from the resolvent at level $\alpha_{k-1}$ to the weighted residual at level $\alpha_{k_*}$, we write
\[
(\alpha_{k-1}I+A)^{-1}
=(\alpha_{k-1}I+A)^{-1}\,\alpha_{k_*}^{-1}(\alpha_{k_*}I+A)\cdot\alpha_{k_*}(\alpha_{k_*}I+A)^{-1}
\]
and estimate, using $\alpha_k\le\alpha_{k-1}\le\alpha_{k_*-1}\le\mu\alpha_{k_*}$ and \eqref{eq:resolvent},
\[
\bigl\|(\alpha_{k-1}I+A)^{-1}\alpha_{k_*}^{-1}(\alpha_{k_*}I+A)\bigr\|
\le\frac{1}{\alpha_{k-1}}+\frac{1}{\alpha_{k_*}}
\le\frac{1}{\alpha_k}+\frac{\mu}{\alpha_{k_*-1}}
\le\frac{1+\mu}{\alpha_k}\le\frac{2(\mu+1)}{\alpha_k} .
\]
Hence
\begin{equation}\label{eq:J1}
\|J_1\|\le\frac{2(\mu+1)}{\alpha_k}\,\bigl\|\alpha_{k_*}(\alpha_{k_*}I+A)^{-1}(F(x_{k_*})-y)\bigr\| .
\end{equation}
Combining \eqref{eq:xkstar-xk}--\eqref{eq:J1} with $\|x_{k_*}-x^\dg\|\le\|x_{k_*}-x_k\|+\|x_k-x^\dg\|$, absorbing the factor of $\|x_{k_*}-x^\dg\|$ into the left-hand side (which is possible by the hypothesis $C_{15}>0$), and using Lemma~\ref{lem:nf-stability} to bound $\|x_{k-1}-x^\dg\|\le\|x_k-x^\dg\|/C_{11}$, we arrive at \eqref{eq:oracle-pre}.
\end{proof}

\subsection{Obstructions to a full oracle inequality}\label{subsec:obstruction}

In the a~posteriori setting of \cite{mahale-singh}, Proposition~\ref{thm:oracle-pre} is completed to a full oracle inequality
\begin{equation}\label{eq:oracle-target}
\|x_{k_*}^\delta-x^\dg\|\ \le\ \xi\,\inf_{k\ge0}\Bigl\{\|x_k-x^\dg\|+\frac{(\sqrt{\delta_*}+\sqrt{\delta})^2}{\alpha_k}\Bigr\}
\end{equation}
by two further estimates for the noise-free weighted residual $a_k:=\|\alpha_k(\alpha_kI+A)^{-1}(F(x_k)-y)\|$: an upper bound $a_{k_\delta}\le C\delta$ at the stopping index and a lower bound $a_k\ge c\,\delta$ for $k<k_\delta$. In \cite{mahale-singh} both estimates are immediate consequences of the a~posteriori rule \eqref{eq:msrule}, which directly controls the weighted residual on both sides of the stopping index. We record here, explicitly, why the corresponding estimates are not available for the heuristic rule under the present assumptions; an earlier version of this manuscript overlooked both points, and we thank the referee for insisting on them.

\begin{itemize}
\item[(a)] \emph{No index comparison for residuals.} The minimizing property of $k_*$ yields only the scaled comparison $\delta_*/\alpha_{k_*}\le\|F(x_k^\delta)-y^\delta\|/\alpha_k$ for admissible $k$. Passing from $k_*$ to the a~priori index $k_*^\delta\ge k_*$ therefore introduces the factor $\alpha_{k_*}/\alpha_{k_*^\delta}$, which by \eqref{eq:kdstar} may be as large as a constant multiple of $\delta^{-1/2}$; this is exactly the content of the upper bound in Corollary~\ref{cor:deltastar}, and it cannot be improved to $\delta_*\le C\delta$ without additional information, since no monotonicity of the residuals $k\mapsto\|F(x_k^\delta)-y^\delta\|$ has been (or, we believe, can be) established under Assumption~\ref{ass:nonlinearity} alone.
\item[(b)] \emph{No lower bound for the weighted residual.} Assumption~\ref{ass:noise} controls the unweighted residual from below, $\|F(x_k^\delta)-y^\delta\|\ge\kappa\delta$. The operator $\alpha_k(\alpha_kI+A)^{-1}$ appearing in $a_k$ is a self-adjoint contraction whose spectrum reaches down to $\alpha_k/(\alpha_k+\|A\|)$, so it admits no uniform lower bound; a lower bound for $a_k$ would require information about the spectral localization of the residual, which the noise condition does not provide.
\end{itemize}

Establishing \eqref{eq:oracle-target} for the heuristic rule, or proving that it fails, therefore remains an open problem; Proposition~\ref{thm:oracle-pre} is the part of the argument that does carry over, and we state it here in the hope that it will be useful for future work in this direction.

We now derive convergence rates under H\"older-type source conditions. As explained in Remark~\ref{rem:p-usage}, the error estimate of Theorem~\ref{thm:rate-range} has been established under the range condition \eqref{eq:range}; since Assumption~\ref{ass:source} with $p\ge1$ implies the range condition, the following result is a direct consequence; no oracle inequality is needed for it.

\begin{theorem}\label{thm:hoelder}
Suppose Assumption~\ref{ass:source} holds with some $p\ge1$ such that $w:=A^{p-1}v\in\N(A)^\perp$ (for $p=1$, $w=v$ and this is \eqref{eq:range}), and let the assumptions of Theorem~\ref{thm:rate-range}, including the horizon condition $K_\infty\ge\hat k$, be satisfied with $v$ replaced by $w$. If $k_*\ge1$, then
\begin{equation}\label{eq:hoelder-rate}
\|x_{k_*}^\delta-x^\dg\|\ \le\ C_8(w)\,\bigl(\sqrt{\delta_*}+\sqrt{\delta}\bigr),
\end{equation}
where $C_8(w)$ denotes the constant of Theorem~\ref{thm:rate-range} formed with $w$ in place of $v$.
\end{theorem}

\begin{proof}
Since $p\ge1$ we may write
\[
e_0=A^pv=A\bigl(A^{p-1}v\bigr)=Aw,\qquad w=A^{p-1}v\in\N(A)^\perp ,
\]
so the range condition \eqref{eq:range} holds with $v$ replaced by $w$. All hypotheses of Theorem~\ref{thm:rate-range} are satisfied by assumption (with the constants formed from $w$), and \eqref{eq:hoelder-rate} is precisely the estimate \eqref{eq:mainrate} of that theorem.

\end{proof}

\begin{remark}\label{rem:plessone}
Two comments on Theorem~\ref{thm:hoelder} are in order.

(i) For $p\ge1$ the rate \eqref{eq:hoelder-rate} coincides with the rate of Theorem~\ref{thm:rate-range}; higher smoothness $p>1$ does not improve the exponent. This reflects the well-known \emph{saturation} of Lavrentiev-type regularization at the qualification corresponding to $p=1$ in the present error decomposition, and is consistent with the a~posteriori results in \cite{mahale-nair09,mahale-singh}.

(ii) For $p\in(0,1)$ a balancing argument based on an oracle inequality of the form \eqref{eq:oracle-target} would formally suggest the rate $O\bigl((\sqrt{\delta_*}+\sqrt{\delta})^{2p/(1+p)}\bigr)$. However, no oracle inequality is available here (Subsection~\ref{subsec:obstruction}), and in addition the auxiliary estimates of Sections~\ref{sec:wd} and~\ref{sec:rates} (in particular Lemmas~\ref{lem:wd-iter}, \ref{lem:nf-stability} and Proposition~\ref{prop:DE}) make essential use of the representation $e_0=Av$. Establishing convergence rates for the heuristic rule under Assumption~\ref{ass:source} with $p<1$, without assuming \eqref{eq:range}, remains an open problem; we expect that this requires different techniques, cf.\ the discussion in \cite{jin-wang} for the Gauss--Newton method.
\end{remark}

\section{Numerical examples}\label{sec:numerics}

In this section we illustrate the performance of Rule~\ref{rule:heuristic} on one-, two-, and three-dimensional nonlinear problems with genuinely monotone forward operators, and we compare it with the a~posteriori stopping rule of discrepancy type
\begin{equation}\label{eq:dp}
k_\delta:=\min\{k\ge0:\ \|F(x_k^\delta)-y^\delta\|\le\tau\delta\},
\end{equation}
which requires knowledge of the noise level $\delta$ and of a suitable constant $\tau$, as well as with the weighted a~posteriori rule \eqref{eq:mnrule} analyzed in the author's earlier work \cite{mahale-singh}, i.e.,
\begin{equation}\label{eq:msrule}
k_\delta:=\min\bigl\{k\ge0:\ \bigl\|\alpha_k\bigl(A_k^\delta+\alpha_kI\bigr)^{-1}\bigl(F(x_k^\delta)-y^\delta\bigr)\bigr\|\le\tau\delta\bigr\},
\end{equation}
used there with $\tau=1.9$, which we adopt here as well. All three examples are instances of the following general class of monotone operators. Let $\Omega\subset\mathbb{R}^d$, $d\le3$, be a bounded domain, let $A:L^2(\Omega)\to L^2(\Omega)$ be a bounded linear operator which is also bounded as a map from $L^2(\Omega)$ into $C(\overline\Omega)$, and let $\varphi\in C^2(\mathbb{R})$ be non-decreasing. Define
\begin{equation}\label{eq:F-general}
F(x):=A^*\varphi(Ax),\qquad x\in L^2(\Omega),
\end{equation}
where $\varphi(Ax)$ denotes the superposition $(\varphi(Ax))(t)=\varphi\bigl((Ax)(t)\bigr)$.

\begin{proposition}\label{prop:general-class}
The operator $F$ in \eqref{eq:F-general} is well defined and monotone on $L^2(\Omega)$, Fr\'echet differentiable with
\begin{equation}\label{eq:F-derivative-general}
F'(x)h=A^*\bigl(\varphi'(Ax)\,Ah\bigr),\qquad h\in L^2(\Omega),
\end{equation}
and $F'$ is Lipschitz continuous on bounded subsets of $L^2(\Omega)$. In particular, Assumption~\ref{ass:nonlinearity} holds locally after the normalization discussed below \eqref{eq:resolvent}. We emphasize that $A$ is not required to be self-adjoint.
\end{proposition}

\begin{proof}
Well-definedness follows since $Ax\in C(\overline\Omega)$, so $\varphi(Ax)\in C(\overline\Omega)\subset L^2(\Omega)$. For monotonicity, let $x,\tilde x\in L^2(\Omega)$ and set $u:=Ax$, $\tilde u:=A\tilde x$. Then
\[
\ip{F(x)-F(\tilde x)}{x-\tilde x}
=\ip{\varphi(u)-\varphi(\tilde u)}{u-\tilde u}
=\int_\Omega\bigl(\varphi(u(t))-\varphi(\tilde u(t))\bigr)\bigl(u(t)-\tilde u(t)\bigr)\,dt\ \ge\ 0,
\]
because $\varphi$ is non-decreasing, so the integrand is pointwise nonnegative. (No pointwise ordering of $x$ and $\tilde x$ and no self-adjointness of $A$ are required; the adjoint $A^*$ in \eqref{eq:F-general} moves onto the second argument of the inner product.)

Since $A:L^2(\Omega)\to C(\overline\Omega)$ is bounded and $\varphi\in C^2(\mathbb{R})$, the Nemytskii operator $u\mapsto\varphi(u)$ is Fr\'echet differentiable from $C(\overline\Omega)$ to $C(\overline\Omega)$, and the composition with the bounded linear maps $A$ and $A^*$ is Fr\'echet differentiable from $L^2(\Omega)$ to $L^2(\Omega)$ with derivative \eqref{eq:F-derivative-general}. For $x,z$ in a ball of radius $\rho$ in $L^2(\Omega)$, the values of $Ax$ and $Az$ lie in a compact interval $I_\rho\subset\mathbb{R}$, and with $M_\rho:=\sup_{I_\rho}|\varphi''|$,
\[
\|F'(x)-F'(z)\|\le\|A^*\|\,\|A\|\,\bigl\|\varphi'(Ax)-\varphi'(Az)\bigr\|_{C(\overline\Omega)}
\le M_\rho\,\|A^*\|\,\|A\|\,\|A\|_{L^2\to C}\,\|x-z\| ,
\]
which is the asserted Lipschitz property.
\end{proof}

The three examples below take, respectively: $A$ the (self-adjoint) Green's operator of $-u''$ on $(0,1)$ with $\varphi(u)=u^3$; $A$ the (self-adjoint) solution operator of a variable-coefficient diffusion problem on the unit square with $\varphi(u)=u+u^3$; and $A$ the (non-self-adjoint) solution operator of a convection--diffusion problem on the unit cube with $\varphi(u)=u+u^3$. In the self-adjoint cases $A^*=A$ and \eqref{eq:F-general} reduces to $F(x)=A\varphi(Ax)$.

In all experiments the noisy data is $y^\delta=y+\delta\chi$, where $\chi$ is the constant function on $\Omega$ normalized to $\|\chi\|_{L^2(\Omega)}=1$, so that $\|y-y^\delta\|=\delta$ exactly; the initial guess satisfies the range condition \eqref{eq:range},
\[
x_0=x^\dg+F'(x^\dg)v,
\]
where $v$ is a fixed smooth function (specified in each subsection) scaled so that $\|x_0-x^\dg\|\approx0.8$; the regularization parameters are $\alpha_k=\alpha_0(1/2)^k$ (as noted after Assumption~\ref{ass:nonlinearity}, the scaling condition \eqref{eq:scaling} is a normalization that can always be achieved by rescaling the equation; the values of $\alpha_0$ used below correspond to such rescalings); and the errors reported in the tables are $\|x_k^\delta-x^\dg\|_{L^2(\Omega)}$ at the respective stopping index. For the discrepancy rule \eqref{eq:dp} we use $\tau=1.9$, $\tau=0.8$, and $\tau=12$ (note that $\tau=0.8<1$ violates the standard requirement $\tau>1$ and is included to illustrate the effect of an inappropriate parameter, which is precisely the situation the heuristic rule is designed to avoid), and for the rule \eqref{eq:msrule} we use $\tau=1.9$ as in \cite{mahale-singh}. The computational horizon of Rule~\ref{rule:heuristic} (cf.\ Remark~\ref{rem:retrospective}) is $K=25$ in the one-dimensional and $K=24$ in the two- and three-dimensional experiments; in every run reported below, $\theta(k,y^\delta)$ was strictly increasing over the final third of the window, and enlarging $K$ did not change $k_*$. All linear systems in \eqref{eq:iter} are solved by a direct (LU) factorization of the dense discretized operator $F'(x_k^\delta)+\alpha_kI$. Following a referee's suggestion, we separate strictly two kinds of experiments.
\emph{(a) Matched-model experiments (Tables~\ref{tab:results}--\ref{tab:results3d} and~\ref{tab:noise}, Figures~\ref{fig:theta}--\ref{fig:3d}):} the exact data is $y=F_h(x^\dg)$, generated by the same discrete forward operator $F_h$ used in the reconstruction. This is deliberate: the theorems of Sections~\ref{sec:wd}--\ref{sec:rates} concern one fixed operator and its exact data, and only the matched setting tests them; the noise variable $\delta$ then coincides exactly with the total data error.
\emph{(b) Model-mismatch stress tests (Table~\ref{tab:mismatch}, Subsection~\ref{subsec:stress}):} the data comes from a different, more accurate discrete model: the closed form \eqref{eq:exact-data} in 1D (quadrature mismatch $\approx1.3\times10^{-7}$), and fine-grid computations in 2D and 3D (grids $63\times63$ and $23\times23\times23$, i.e., half the mesh size, with sparse LU solves; mismatch $\|F_h(x^\dg)-y\|\approx1.2\times10^{-6}$ and $\approx1.1\times10^{-5}$, respectively). The coarse-grid nodes coincide with every second fine-grid node, and the fine-to-coarse operation is the restriction to these coincident coarse nodes. In these runs the reported $\delta$ is the norm of the added perturbation only; the total discrete data error additionally contains the model mismatch, so these tables probe robustness rather than the $\delta\to0$ theorems.
Reproducibility details: all norms and inner products are the discrete $L^2$ products with quadrature weights ($w_i=h$, with $w_i=h/2$ at the endpoints, in 1D; the uniform weights $h^2$ and $h^3$ in 2D/3D). In 1D the discrete operator $A=GW$ (Nystr\"om) is self-adjoint with respect to this weighted inner product because the kernel matrix $G$ is symmetric; in 2D/3D the weights are uniform multiples of the identity, so the weighted adjoint coincides with the Euclidean matrix transpose, which is what the implementation uses ($F_h(x)=A^{\mathsf T}\varphi(Ax)$); no separate mass matrix is involved. The random seed for the noise realizations is $2026$; the complete scripts reproducing every table and figure accompany the paper. Finally, along every trajectory we monitor the \emph{empirical residual ratio}
\[
\rho_{\mathrm{emp}}:=\min_{0\le k\le K}\frac{\|F_h(x_k^\delta)-y^\delta\|}{\delta};
\]
in the matched setting (a) this is a \emph{finite-window empirical diagnostic} for the constant $\kappa$ of Assumption~\ref{ass:noise}: it probes the assumption only for $0\le k\le K$, whereas the assumption is stated over the complete orbit up to $k_\infty$, and a positive finite-window minimum cannot exclude a smaller ratio at a later iteration. It is therefore a diagnostic, not a numerical verification, of the assumption. In the mismatched setting (b) the assumption does not literally apply (the data is not exact data of $F_h$) and $\rho_{\mathrm{emp}}$ is reported merely as an empirical ratio; see Subsections~\ref{subsec:robust} and~\ref{subsec:stress}.

\subsection{A one-dimensional integral equation}\label{subsec:1d}

Let $K:L^2(0,1)\to L^2(0,1)$ denote the integral operator
\[
(Kx)(t)=\int_0^1G(t,s)\,x(s)\,ds,\qquad
G(t,s)=\begin{cases} s(1-t), & 0\le s\le t\le1,\\ t(1-s), & 0\le t\le s\le1,\end{cases}
\]
i.e., $G$ is the Green's function of $-u''$ with homogeneous Dirichlet boundary conditions, so that $K$ is compact, self-adjoint and positive semidefinite. Define
\begin{equation}\label{eq:F-example}
F(x):=K\bigl((Kx)^3\bigr),\qquad x\in L^2(0,1).
\end{equation}
Note that $F$ is well defined on all of $L^2(0,1)$: $K$ maps $L^2(0,1)$ boundedly into $C[0,1]$, so $(Kx)^3\in C[0,1]\subset L^2(0,1)$.

This is the operator \eqref{eq:F-general} with $A=K$ (self-adjoint) and $\varphi(u)=u^3$, so Proposition~\ref{prop:general-class} applies: $F$ is monotone on $L^2(0,1)$, Fr\'echet differentiable with $F'(x)h=K\bigl(3(Kx)^2\,Kh\bigr)$, and $F'$ is Lipschitz continuous on bounded sets.

\begin{remark}
The example used in earlier versions of this work and in parts of the literature, namely $\tilde F(x)=K(x^3)$, is not monotone on $L^2(0,1)$: for that operator one only has
$\ip{\tilde F(x)-\tilde F(\tilde x)}{x-\tilde x}=\ip{K(x^3-\tilde x^3)}{x-\tilde x}$, and since $K$ mixes function values at different points, the pointwise nonnegativity of $(x^3-\tilde x^3)(x-\tilde x)$ does not imply nonnegativity of the inner product. The sandwiched construction \eqref{eq:F-example} repairs this defect.
\end{remark}

For the experiment we choose the exact solution
\[
x^\dg(t)=4\sin(\pi t).
\]
Because $\sin(j\pi t)$ is an eigenfunction of $K$ with eigenvalue $(j\pi)^{-2}$, the exact data is available in closed form:
\[
Kx^\dg=\frac{4}{\pi^2}\sin(\pi t),\qquad
(Kx^\dg)^3=\frac{16}{\pi^6}\bigl(3\sin(\pi t)-\sin(3\pi t)\bigr),
\]
and hence
\begin{equation}\label{eq:exact-data}
y(t)=F(x^\dg)(t)=\frac{16}{\pi^8}\Bigl(3\sin(\pi t)-\frac{1}{9}\sin(3\pi t)\Bigr).
\end{equation}
Here $v(t)=c_1\sin(\pi t)$ with $c_1=300$, which yields $\|x_0-x^\dg\|\approx0.805$, and $\alpha_0=10^{-4}$. For the discretization, the interval $[0,1]$ is divided into $300$ equal subintervals and all integrals are approximated by the trapezoidal rule (Nystr\"om discretization $F_h$); the linear systems in \eqref{eq:iter} are solved directly. To remove any ambiguity about the data source: the matched-model experiments of this subsection (Table~\ref{tab:results}, Figures~\ref{fig:theta}--\ref{fig:reconstructions}, and the 1D rows of Table~\ref{tab:noise}) use $y=F_h(x^\dg)$, computed by the same trapezoidal Nystr\"om operator used for the reconstruction; the closed form \eqref{eq:exact-data} is evaluated only in the 1D rows of the mismatch stress test (Table~\ref{tab:mismatch}), where the quadrature error $\approx1.3\times10^{-7}$ constitutes the deliberate mismatch, and otherwise serves as an analytical reference. The matched results for $\delta$ between $10^{-4}$ and $10^{-7}$ are reported in Table~\ref{tab:results}.

\begin{table}[htb]
\centering
\caption{Matched-model results (setting (a)) for the one-dimensional example \eqref{eq:F-example}: stopping indices and errors for the discrepancy principle \eqref{eq:dp} with various $\tau$, for the weighted a~posteriori rule \eqref{eq:msrule} of \cite{mahale-singh} with $\tau=1.9$, and for the heuristic rule (Rule~\ref{rule:heuristic}); the last column reports the optimal (oracle) error $\min_{0\le k\le K}\|x_k^\delta-x^\dg\|$ along the computed trajectory.}
\label{tab:results}
\footnotesize
\begin{tabular}{c cc cc cc cc ccc c}
\toprule
& \multicolumn{6}{c}{discrepancy principle \eqref{eq:dp}} & \multicolumn{2}{c}{rule \eqref{eq:msrule}} & \multicolumn{3}{c}{Rule~\ref{rule:heuristic}} & \\
\cmidrule(lr){2-7}\cmidrule(lr){8-9}\cmidrule(lr){10-12}
& \multicolumn{2}{c}{$\tau=1.9$} & \multicolumn{2}{c}{$\tau=0.8$} & \multicolumn{2}{c}{$\tau=12$} & \multicolumn{2}{c}{$\tau=1.9$} & & & & optimal\\
\cmidrule(lr){2-3}\cmidrule(lr){4-5}\cmidrule(lr){6-7}\cmidrule(lr){8-9}
$\delta$ & $k_\delta$ & error & $k_\delta$ & error & $k_\delta$ & error & $k_\delta$ & error & $k_*$ & $\delta_*$ & error & error\\
\midrule
$1.00\times 10^{-4}$ & 2 & 0.7928 & 3 & 1.4887 & 1 & 0.4476 & 0 & 0.8054 & 2 & $8.02\times 10^{-5}$ & 0.7928 & 0.4476 \\
$5.00\times 10^{-5}$ & 2 & 0.3977 & 3 & 0.7469 & 2 & 0.3977 & 0 & 0.8054 & 3 & $2.74\times 10^{-5}$ & 0.7469 & 0.2725 \\
$1.00\times 10^{-5}$ & 4 & 0.2813 & 5 & 0.5300 & 2 & 0.0822 & 1 & 0.1813 & 3 & $2.06\times 10^{-5}$ & 0.1500 & 0.0822 \\
$5.00\times 10^{-6}$ & 5 & 0.2653 & 6 & 0.5001 & 3 & 0.0752 & 1 & 0.1773 & 3 & $2.03\times 10^{-5}$ & 0.0752 & 0.0446 \\
$1.00\times 10^{-6}$ & 7 & 0.1889 & 8 & 0.3562 & 4 & 0.0283 & 2 & 0.0210 & 4 & $1.00\times 10^{-5}$ & 0.0283 & 0.0160 \\
$5.00\times 10^{-7}$ & 8 & 0.1782 & 9 & 0.3362 & 5 & 0.0266 & 2 & 0.0197 & 4 & $1.00\times 10^{-5}$ & 0.0144 & 0.0093 \\
$1.00\times 10^{-7}$ & 10 & 0.1270 & 12 & 0.4520 & 8 & 0.0357 & 3 & 0.0056 & 4 & $1.00\times 10^{-5}$ & 0.0039 & 0.0039 \\
 
\bottomrule
\end{tabular}
\end{table}

Figure~\ref{fig:theta} shows, for $\delta=10^{-7}$, the residual norms $\|F(x_k^\delta)-y^\delta\|$ and the values $\theta(k,y^\delta)$ as functions of $k$; the minimizer $k_*$ of $\theta$ is clearly visible. Figure~\ref{fig:reconstructions} displays the reconstructions obtained with the heuristic rule and with the discrepancy principle for the three values of $\tau$.

\begin{figure}[htb]
\centering
\includegraphics[width=\textwidth]{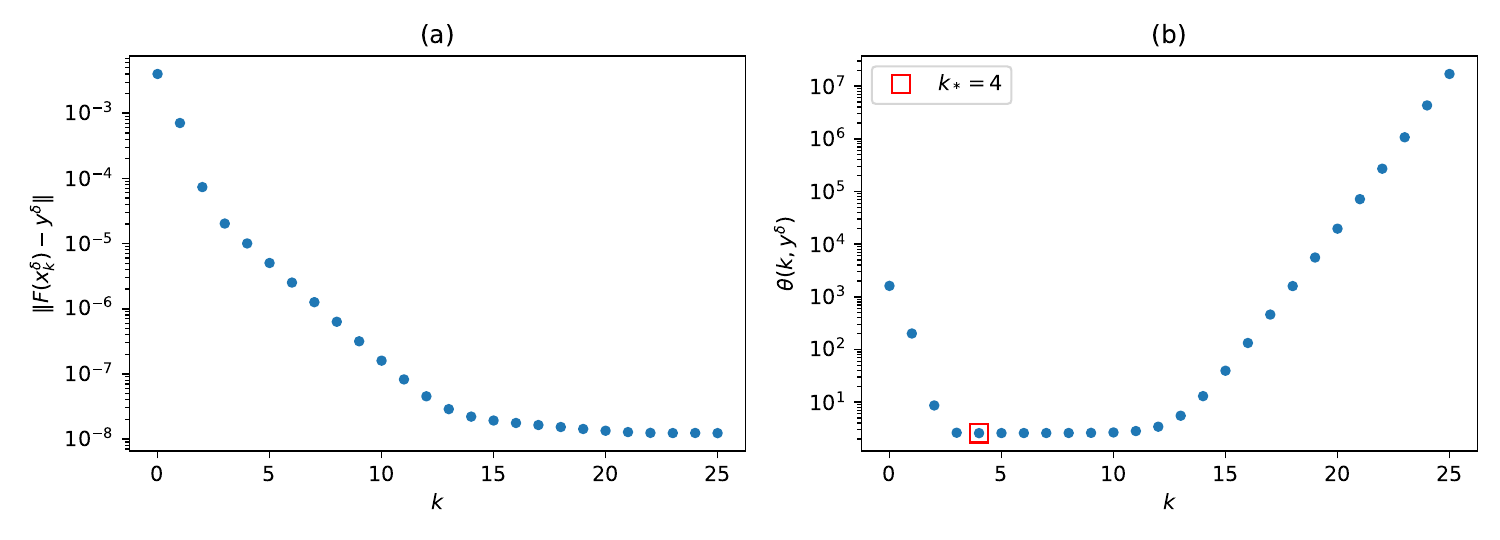}
\caption{Computational results for $\delta=10^{-7}$: (a) $\|F(x_k^\delta)-y^\delta\|$ versus $k$; (b) $\theta(k,y^\delta)$ versus $k$; the square marks the minimizer $k_*$ selected by Rule~\ref{rule:heuristic}.}
\label{fig:theta}
\end{figure}

\begin{figure}[htb]
\centering
\includegraphics[width=\textwidth]{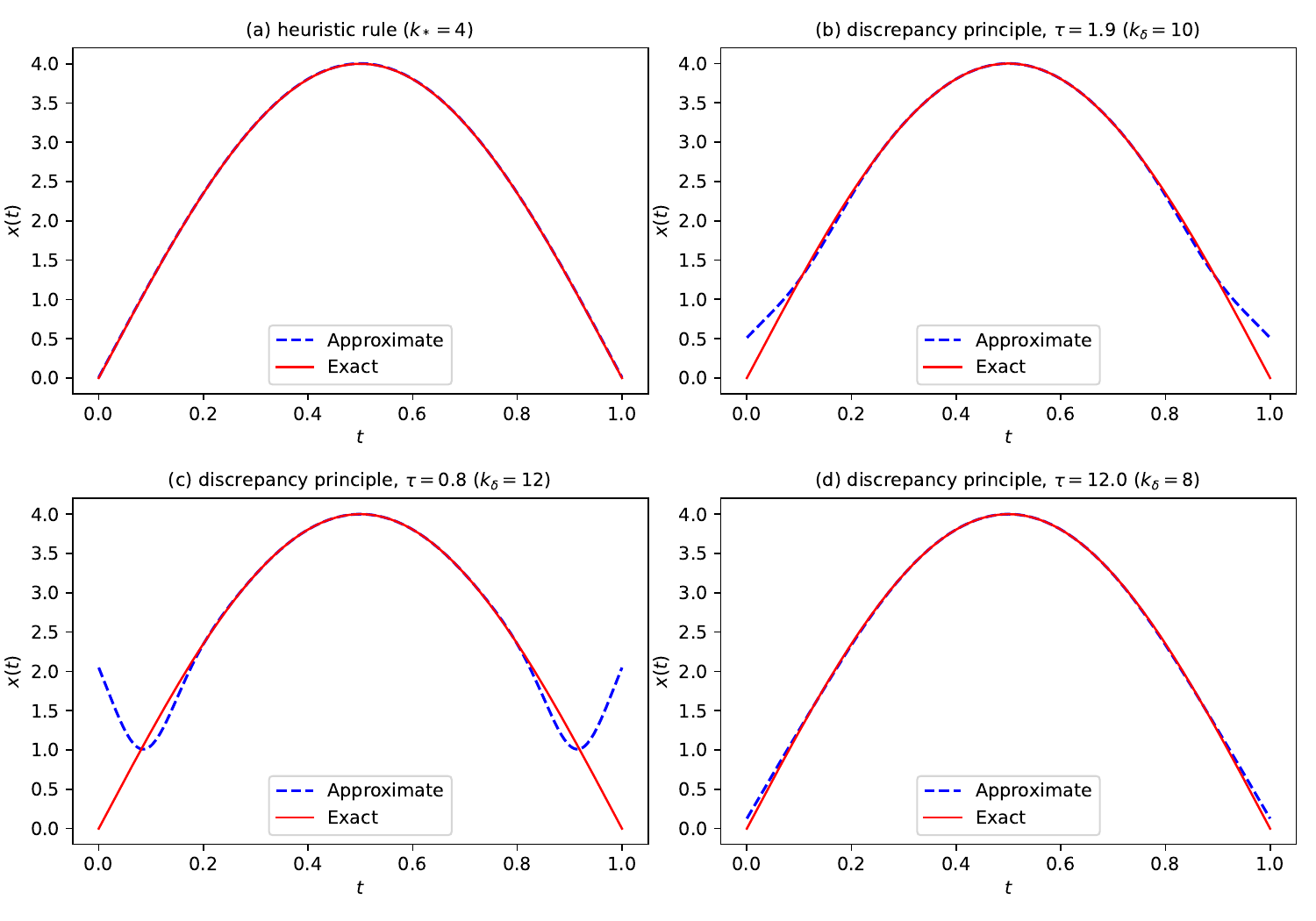}
\caption{Reconstructions for $\delta=10^{-7}$: (a) heuristic rule (Rule~\ref{rule:heuristic}); (b)--(d) discrepancy principle \eqref{eq:dp} with $\tau=1.9$, $\tau=0.8$ and $\tau=12$, respectively.}
\label{fig:reconstructions}
\end{figure}

\subsection{A two-dimensional variable-coefficient diffusion problem}\label{subsec:2d}

We next consider \eqref{eq:F-general} on the unit square $\Omega=(0,1)^2$ with the self-adjoint operator $A=K_a$, where $K_a$ denotes the solution operator of the variable-coefficient elliptic problem
\begin{equation}\label{eq:varcoef}
-\nabla\cdot\bigl(a(s,t)\nabla w\bigr)=x \ \text{ in }\Omega,\qquad w=0 \ \text{ on }\partial\Omega,
\qquad a(s,t):=1+2s+t ,
\end{equation}
and with the nonlinearity $\varphi(u)=u+u^3$, i.e., $F(x)=K_a\bigl(K_ax+(K_ax)^3\bigr)$. Since $a$ is smooth and uniformly positive, $K_a$ is compact, self-adjoint, positive definite, and maps $L^2(\Omega)$ boundedly into $C(\overline\Omega)$, so Proposition~\ref{prop:general-class} applies. In contrast to the one-dimensional example, the operator has no closed-form eigensystem, and we choose a non-symmetric, sign-changing exact solution consisting of two Gaussian bumps,
\[
x^\dg(s,t)=15\,e^{-50\left((s-0.3)^2+(t-0.35)^2\right)}-10\,e^{-60\left((s-0.7)^2+(t-0.65)^2\right)},
\]
with $\|x^\dg\|_{L^2(\Omega)}\approx3.11$; in the matched-model runs of Table~\ref{tab:results2d} the exact data is $y=F_h(x^\dg)$, computed by the discrete forward operator itself (setting (a) above), while the fine-grid data for the stress test of Subsection~\ref{subsec:stress} is generated on the $63\times63$ grid and restricted to the coincident coarse nodes.

For the discretization we use the standard conservative five-point finite difference approximation of $-\nabla\cdot(a\nabla\cdot)$ (with the coefficient evaluated at the cell interfaces) on a uniform grid with $31\times31$ interior points ($h=1/32$); the resulting matrix is symmetric positive definite, so the discrete operator inherits the monotonicity of the continuous one. The initial guess uses $v$ proportional to $x^\dg$, scaled so that $\|x_0-x^\dg\|\approx0.8$, and we take $\alpha_0=10^{-5}$. The results for noise levels between $10^{-5}$ and $10^{-8}$ are reported in Table~\ref{tab:results2d}, and Figure~\ref{fig:2d} displays the exact solution, the reconstructions obtained by the heuristic rule and by the discrepancy principle with $\tau=1.9$, and the pointwise error of the heuristic reconstruction for $\delta=5\times10^{-8}$.

\begin{table}[htb]
\centering
\caption{Matched-model results (setting (a)) for the two-dimensional variable-coefficient example: stopping indices and errors for the discrepancy principle \eqref{eq:dp} with various $\tau$, for the weighted a~posteriori rule \eqref{eq:msrule} of \cite{mahale-singh} with $\tau=1.9$, and for the heuristic rule (Rule~\ref{rule:heuristic}); the last column reports the optimal (oracle) error $\min_{0\le k\le K}\|x_k^\delta-x^\dg\|$ along the computed trajectory.}
\label{tab:results2d}
\footnotesize
\begin{tabular}{c cc cc cc cc ccc c}
\toprule
& \multicolumn{6}{c}{discrepancy principle \eqref{eq:dp}} & \multicolumn{2}{c}{rule \eqref{eq:msrule}} & \multicolumn{3}{c}{Rule~\ref{rule:heuristic}} & \\
\cmidrule(lr){2-7}\cmidrule(lr){8-9}\cmidrule(lr){10-12}
& \multicolumn{2}{c}{$\tau=1.9$} & \multicolumn{2}{c}{$\tau=0.8$} & \multicolumn{2}{c}{$\tau=12$} & \multicolumn{2}{c}{$\tau=1.9$} & & & & optimal\\
\cmidrule(lr){2-3}\cmidrule(lr){4-5}\cmidrule(lr){6-7}\cmidrule(lr){8-9}
$\delta$ & $k_\delta$ & error & $k_\delta$ & error & $k_\delta$ & error & $k_\delta$ & error & $k_*$ & $\delta_*$ & error & error\\
\midrule
$1.00\times 10^{-5}$ & 1 & 0.3690 & 2 & 0.6516 & 1 & 0.3690 & 0 & 0.8000 & 1 & $9.90\times 10^{-6}$ & 0.3690 & 0.3690 \\
$5.00\times 10^{-6}$ & 1 & 0.1868 & 3 & 0.5701 & 1 & 0.1868 & 0 & 0.8000 & 2 & $4.24\times 10^{-6}$ & 0.3259 & 0.1868 \\
$1.00\times 10^{-6}$ & 4 & 0.1971 & 5 & 0.3354 & 1 & 0.0561 & 1 & 0.0561 & 2 & $3.94\times 10^{-6}$ & 0.0676 & 0.0561 \\
$5.00\times 10^{-7}$ & 5 & 0.1677 & 6 & 0.2797 & 2 & 0.0378 & 1 & 0.0475 & 2 & $3.93\times 10^{-6}$ & 0.0378 & 0.0378 \\
$1.00\times 10^{-7}$ & 7 & 0.0910 & 8 & 0.1430 & 4 & 0.0205 & 2 & 0.0212 & 2 & $3.93\times 10^{-6}$ & 0.0212 & 0.0156 \\
$5.00\times 10^{-8}$ & 8 & 0.0715 & 9 & 0.1077 & 5 & 0.0170 & 2 & 0.0206 & 2 & $3.93\times 10^{-6}$ & 0.0206 & 0.0114 \\
$1.00\times 10^{-8}$ & 10 & 0.0308 & 11 & 0.0413 & 8 & 0.0143 & 3 & 0.0109 & 2 & $3.93\times 10^{-6}$ & 0.0205 & 0.0045 \\
 
\bottomrule
\end{tabular}
\end{table}

\begin{figure}[htb]
\centering
\includegraphics[width=0.95\textwidth]{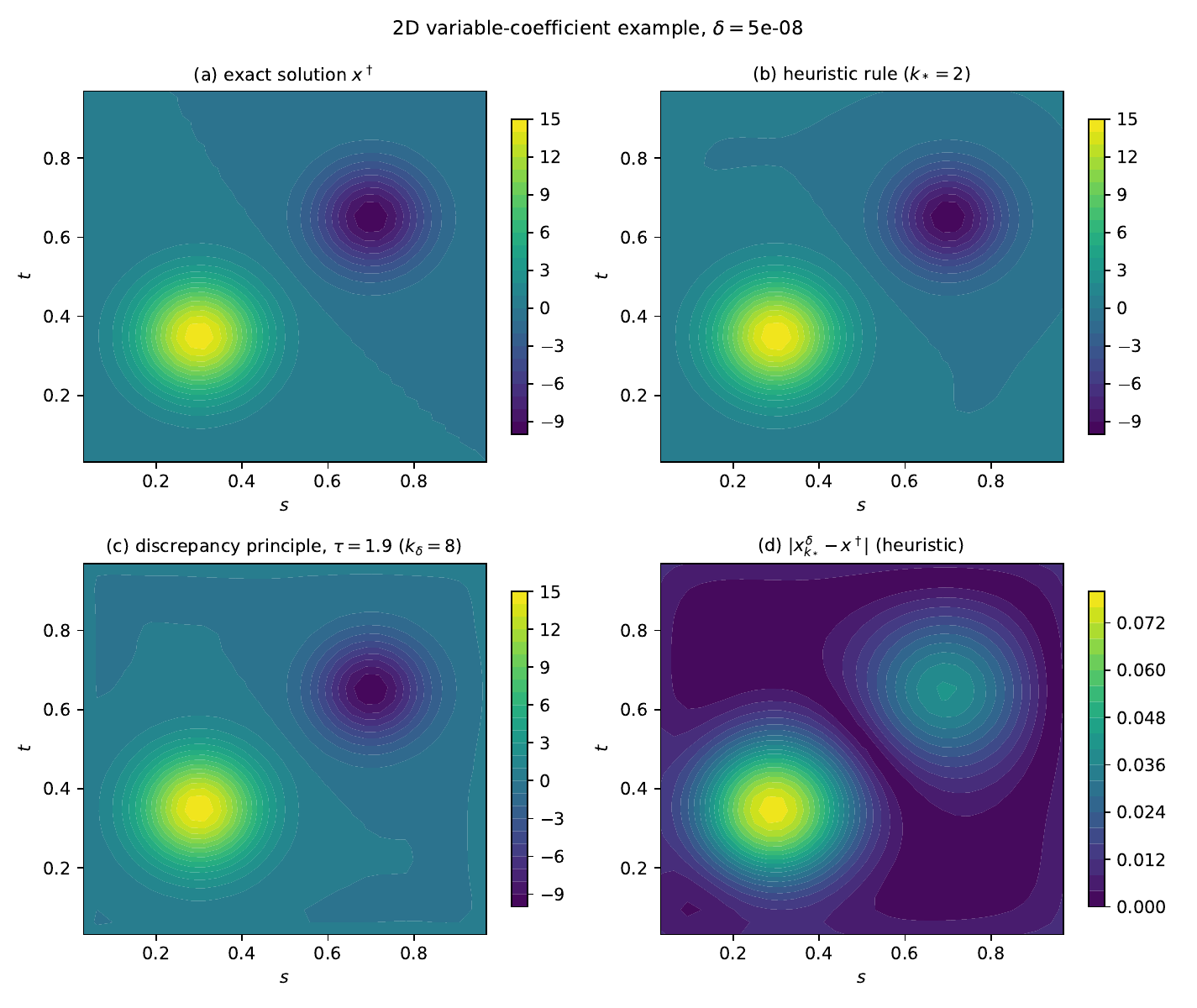}
\caption{Two-dimensional variable-coefficient example with $\delta=5\times10^{-8}$: (a) exact solution; (b) reconstruction by the heuristic rule; (c) reconstruction by the discrepancy principle \eqref{eq:dp} with $\tau=1.9$; (d) pointwise error of the heuristic reconstruction.}
\label{fig:2d}
\end{figure}

\subsection{A three-dimensional convection--diffusion problem}\label{subsec:3d}

Finally, we consider \eqref{eq:F-general} on the unit cube $\Omega=(0,1)^3$ with $A$ the solution operator of the convection--diffusion problem
\begin{equation}\label{eq:convdiff}
-\Delta w+b\cdot\nabla w=x \ \text{ in }\Omega,\qquad w=0 \ \text{ on }\partial\Omega,
\qquad b:=(2,1,-1),
\end{equation}
and again $\varphi(u)=u+u^3$. The operator $A$ is compact and maps $L^2(\Omega)$ boundedly into $C(\overline\Omega)$, but it is not self-adjoint because of the convection term. We emphasize what this example does and does not test: by Proposition~\ref{prop:general-class}, $F(x)=A^*\varphi(Ax)$ is monotone regardless of any symmetry of $A$, and its derivative $F'(x)=A^*\varphi'(Ax)A$ is self-adjoint and nonnegative even though $A$ is not, so the example satisfies Assumption~\ref{ass:nonlinearity}(iv) and lies squarely within the scope of our theory. It demonstrates that the assumption of a self-adjoint derivative does not confine the theory to problems built from symmetric PDE operators; it does not provide a test of genuinely non-self-adjoint derivatives $F'(x)$, which are outside the proved scope of this paper (cf.\ Remark~\ref{rem:selfadjoint}). The exact solution is again a sign-changing pair of Gaussian bumps,
\[
x^\dg(s,t,u)=20\,e^{-40\left((s-0.4)^2+(t-0.4)^2+(u-0.5)^2\right)}
-12\,e^{-50\left((s-0.65)^2+(t-0.6)^2+(u-0.5)^2\right)},
\]
with $\|x^\dg\|_{L^2(\Omega)}\approx1.90$. In the matched-model runs of Table~\ref{tab:results3d} and Figure~\ref{fig:3d} the exact data is $y=F_h(x^\dg)$, computed by the coarse ($11\times11\times11$) discrete forward operator itself (setting (a)); the data for the 3D rows of the stress test in Subsection~\ref{subsec:stress} is generated on the $23\times23\times23$ grid and restricted to the coincident coarse nodes (setting (b)).

The discretization uses the seven-point finite difference Laplacian combined with central differences for the convection term on a uniform grid with $11\times11\times11$ interior points ($h=1/12$; the cell P\'eclet number $\|b\|_\infty h/2=1/12$ is well below one, so central differencing is stable); the discrete adjoint is the matrix transpose. The initial guess uses $v$ proportional to $x^\dg$, scaled so that $\|x_0-x^\dg\|\approx0.8$, and we take $\alpha_0=5\times10^{-6}$. Table~\ref{tab:results3d} reports the results for noise levels between $10^{-5}$ and $10^{-8}$, and Figure~\ref{fig:3d} shows the mid-plane slice $u=1/2$ of the exact solution, of the reconstructions by both rules, and of the pointwise error of the heuristic reconstruction for $\delta=5\times10^{-8}$.

\begin{table}[htb]
\centering
\caption{Matched-model results (setting (a)) for the three-dimensional convection--diffusion example: stopping indices and errors for the discrepancy principle \eqref{eq:dp} with various $\tau$, for the weighted a~posteriori rule \eqref{eq:msrule} of \cite{mahale-singh} with $\tau=1.9$, and for the heuristic rule (Rule~\ref{rule:heuristic}); the last column reports the optimal (oracle) error $\min_{0\le k\le K}\|x_k^\delta-x^\dg\|$ along the computed trajectory.}
\label{tab:results3d}
\footnotesize
\begin{tabular}{c cc cc cc cc ccc c}
\toprule
& \multicolumn{6}{c}{discrepancy principle \eqref{eq:dp}} & \multicolumn{2}{c}{rule \eqref{eq:msrule}} & \multicolumn{3}{c}{Rule~\ref{rule:heuristic}} & \\
\cmidrule(lr){2-7}\cmidrule(lr){8-9}\cmidrule(lr){10-12}
& \multicolumn{2}{c}{$\tau=1.9$} & \multicolumn{2}{c}{$\tau=0.8$} & \multicolumn{2}{c}{$\tau=12$} & \multicolumn{2}{c}{$\tau=1.9$} & & & & optimal\\
\cmidrule(lr){2-3}\cmidrule(lr){4-5}\cmidrule(lr){6-7}\cmidrule(lr){8-9}
$\delta$ & $k_\delta$ & error & $k_\delta$ & error & $k_\delta$ & error & $k_\delta$ & error & $k_*$ & $\delta_*$ & error & error\\
\midrule
$1.00\times 10^{-5}$ & 1 & 0.2397 & 2 & 0.3036 & 1 & 0.2397 & 0 & 0.8000 & 2 & $2.13\times 10^{-6}$ & 0.3036 & 0.2397 \\
$5.00\times 10^{-6}$ & 2 & 0.1518 & 2 & 0.1518 & 1 & 0.1250 & 0 & 0.8000 & 2 & $2.03\times 10^{-6}$ & 0.1518 & 0.1250 \\
$1.00\times 10^{-6}$ & 3 & 0.0362 & 4 & 0.0405 & 2 & 0.0307 & 1 & 0.0486 & 3 & $9.99\times 10^{-7}$ & 0.0362 & 0.0307 \\
$5.00\times 10^{-7}$ & 4 & 0.0203 & 5 & 0.0217 & 2 & 0.0159 & 1 & 0.0443 & 3 & $9.98\times 10^{-7}$ & 0.0182 & 0.0159 \\
$1.00\times 10^{-7}$ & 6 & 0.0045 & 7 & 0.0046 & 3 & 0.0044 & 1 & 0.0430 & 3 & $9.98\times 10^{-7}$ & 0.0044 & 0.0042 \\
$5.00\times 10^{-8}$ & 7 & 0.0023 & 8 & 0.0023 & 4 & 0.0024 & 2 & 0.0051 & 3 & $9.98\times 10^{-7}$ & 0.0031 & 0.0023 \\
$1.00\times 10^{-8}$ & 9 & 0.0005 & 10 & 0.0005 & 7 & 0.0005 & 2 & 0.0049 & 3 & $9.98\times 10^{-7}$ & 0.0025 & 0.0005 \\
 
\bottomrule
\end{tabular}
\end{table}

\begin{figure}[htb]
\centering
\includegraphics[width=0.95\textwidth]{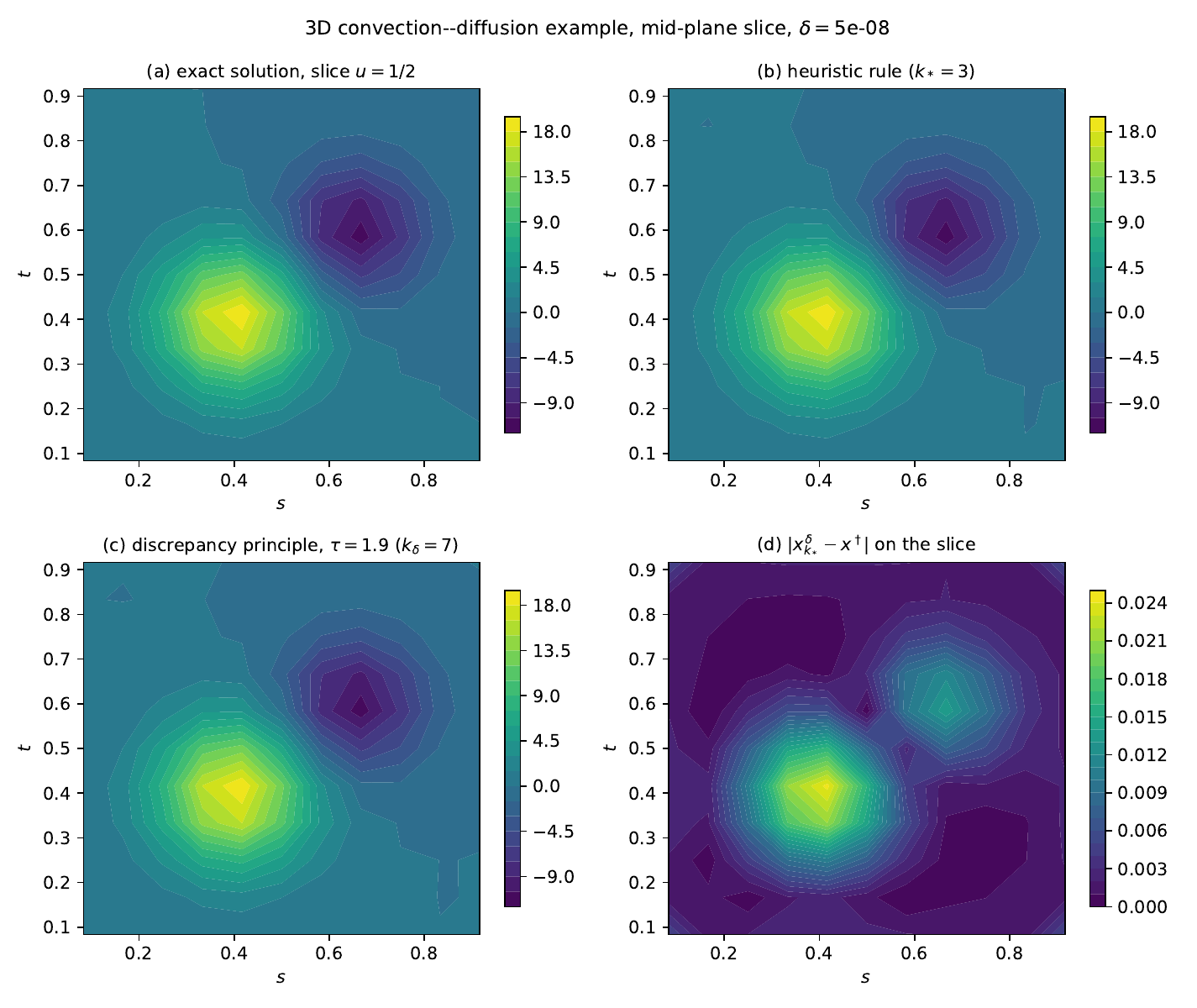}
\caption{Three-dimensional convection--diffusion example with $\delta=5\times10^{-8}$, mid-plane slice $u=1/2$: (a) exact solution; (b) reconstruction by the heuristic rule; (c) reconstruction by the discrepancy principle \eqref{eq:dp} with $\tau=1.9$; (d) pointwise error of the heuristic reconstruction on the slice.}
\label{fig:3d}
\end{figure}

\subsection{Noise models, replications, and diagnostics}\label{subsec:robust}

The experiments above use a fixed, smooth (constant-direction) perturbation, which makes the tables deterministic and reproducible but does not probe robustness with respect to the structure of the noise. Table~\ref{tab:noise} therefore repeats the one-dimensional matched-model experiment with two further noise models, in both cases normalized so that $\|y-y^\delta\|=\delta$ exactly: \emph{white noise}, $y^\delta=y+\delta\,e/\|e\|$ with $e$ a vector of independent standard Gaussian entries ($R=30$ independent realizations per noise level; the table reports mean $\pm$ standard deviation of the errors), and a deterministic \emph{oscillatory} perturbation proportional to $\sin(31\pi t)+\tfrac12\sin(53\pi t)$, which is nearly orthogonal to the smooth ranges of the low-order iterates. The last column reports the smallest observed value of the residual ratio $\rho_{\mathrm{emp}}$ defined in the preamble; since these runs are matched, $\rho_{\mathrm{emp}}$ is a finite-window empirical diagnostic for the constant $\kappa$ of Assumption~\ref{ass:noise} (over $0\le k\le K$; see the preamble).

\begin{table}[htb]
\centering
\caption{One-dimensional matched-model example with random and oscillatory noise. Reconstruction errors for the heuristic rule, the discrepancy principle ($\tau=1.9$), and the weighted rule \eqref{eq:msrule} ($\tau=1.9$); ``optimal'' is the oracle error along the trajectory; for white noise, mean $\pm$ standard deviation over $R=30$ realizations. The last column is the smallest observed value of $\rho_{\mathrm{emp}}$ (a finite-window diagnostic for $\kappa$ in this matched setting).}
\label{tab:noise}
\footnotesize
\begin{tabular}{c l c c c c c}
\toprule
$\delta$ & noise & Rule~\ref{rule:heuristic} & \eqref{eq:dp}, $\tau{=}1.9$ & \eqref{eq:msrule}, $\tau{=}1.9$ & optimal & $\rho_{\mathrm{emp}}$\\
\midrule
$1.00\times 10^{-5}$ & white ($R{=}30$) & 0.3979 $\pm$ 0.0011 & 0.7947 $\pm$ 0.0023 & 0.2018 $\pm$ 0.0002 & 0.2001 $\pm$ 0.0005 & 0.96 \\
$1.00\times 10^{-5}$ & oscillatory & 0.4000 & 0.7999 & 0.2020 & 0.2009 & 0.04 \\
$1.00\times 10^{-6}$ & white ($R{=}30$) & 0.0795 $\pm$ 0.0003 & 0.6325 $\pm$ 0.0034 & 0.0277 $\pm$ 0.0001 & 0.0277 $\pm$ 0.0001 & 0.94 \\
$1.00\times 10^{-6}$ & oscillatory & 0.0800 & 0.6391 & 0.0277 & 0.0277 & 0.14 \\
$1.00\times 10^{-7}$ & white ($R{=}30$) & 0.0084 $\pm$ 0.0000 & 0.5016 $\pm$ 0.0024 & 0.0066 $\pm$ 0.0000 & 0.0066 $\pm$ 0.0000 & 0.90 \\
$1.00\times 10^{-7}$ & oscillatory & 0.0084 & 0.5061 & 0.0067 & 0.0067 & 0.74 \\
 
\bottomrule
\end{tabular}
\end{table}

Three observations are worth recording. First, the behavior of the heuristic rule is stable across noise models and realizations: the standard deviations over the thirty white-noise realizations are two to three orders of magnitude below the means. Second, the diagnostic $\rho_{\mathrm{emp}}$ is comfortably bounded away from zero for white noise ($\rho_{\mathrm{emp}}\ge0.90$ in all such runs), giving empirical support to Assumption~\ref{ass:noise} in that regime; for the oscillatory perturbation at the largest noise level, however, $\rho_{\mathrm{emp}}$ drops to $0.04$ because the iterates can almost fit this structured perturbation, and it is precisely there that the noise condition, and with it our theory, is close to failing, even though the selected index happened to remain reasonable. The same caveat applies, more strongly, to the constant-direction perturbation: in the matched 2D/3D runs of Tables~\ref{tab:results2d}--\ref{tab:results3d} this smooth perturbation is almost fittable by the iterates, with finite-window ratios as small as $\rho_{\mathrm{emp}}\sim10^{-5}$ (2D) and $\sim5\times10^{-8}$ (3D); Assumption~\ref{ass:noise} is thus empirically strained for this noise model even in the matched setting. This is exactly why the white-noise replications, where the assumption is comfortably satisfied, are included: the deterministic tables should be read for reproducibility, the white-noise rows for the empirical standing of the noise condition. Third, we stress an honest limitation: the theoretical smallness conditions of Sections~\ref{sec:wd}--\ref{sec:rates} (on $L\|v\|$, $C_5$, $C_{13}$, etc.)\ are sufficient conditions with conservative constants, and we do not claim that the experiments verify them quantitatively; what the experiments document is the observable behavior of the rule (choice of $k_*$, size of $\delta_*$, reconstruction errors) together with the finite-window diagnostic $\rho_{\mathrm{emp}}$ in the matched setting, which probes Assumption~\ref{ass:noise} but, being restricted to $0\le k\le K$, cannot verify it.

\subsection{Model-mismatch stress tests}\label{subsec:stress}

The experiments of this subsection deliberately step outside the exact setting of the theorems: the data is generated by a more accurate discrete model (setting (b) of the preamble), so the discrete problem sees the added perturbation plus an operator mismatch, while the reported $\delta$ measures the added perturbation only. These runs are therefore not evidence for the $\delta\to0$ theorems; they test robustness when the effective data error is underestimated by the nominal $\delta$, which is the situation in which noise-level-free rules matter most. In the 3D rows the mismatch $\approx1.1\times10^{-5}$ dominates every nominal noise value, so those rows are model-error robustness experiments in their entirety.

\begin{table}[htb]
\centering
\caption{Model-mismatch stress tests (setting (b)): data from a more accurate discrete model (closed-form/quadrature in 1D; fine grids $63\times63$ and $23\times23\times23$ in 2D/3D, restricted to the coincident coarse nodes), reconstruction with the coarse operator, constant-direction added noise of norm $\delta$. Errors for the heuristic rule, the discrepancy principle ($\tau=1.9$), the weighted rule \eqref{eq:msrule} ($\tau=1.9$), and the oracle. The last column is the empirical residual ratio $\rho_{\mathrm{emp}}$, which in this mismatched setting is not a diagnostic for Assumption~\ref{ass:noise}.}
\label{tab:mismatch}
\footnotesize
\begin{tabular}{l c c c c c c}
\toprule
example & $\delta$ (added) & Rule~\ref{rule:heuristic} & \eqref{eq:dp}, $\tau{=}1.9$ & \eqref{eq:msrule}, $\tau{=}1.9$ & optimal & $\rho_{\mathrm{emp}}$\\
\midrule
1D (quadrature) & $1.00\times 10^{-5}$ & 0.1501 & 0.2815 & 0.1813 & 0.0823 & 0.067 \\
1D (quadrature) & $1.00\times 10^{-6}$ & 0.0285 & 0.1893 & 0.0210 & 0.0161 & 0.11 \\
1D (quadrature) & $1.00\times 10^{-7}$ & 0.0040 & 0.1278 & 0.0056 & 0.0040 & 0.12 \\
2D (fine-grid data) & $1.00\times 10^{-6}$ & 0.0655 & 0.1965 & 0.0480 & 0.0480 & $9.05\times 10^{-6}$ \\
2D (fine-grid data) & $1.00\times 10^{-7}$ & 0.0147 & 0.0965 & 0.0147 & 0.0147 & $1.32\times 10^{-5}$ \\
2D (fine-grid data) & $1.00\times 10^{-8}$ & 0.0138 & 0.0510 & 0.0244 & 0.0138 & $9.66\times 10^{-5}$ \\
3D (fine-grid data) & $1.00\times 10^{-5}$ & 0.3364 & 0.2547 & 0.8000 & 0.2547 & $5.74\times 10^{-8}$ \\
3D (fine-grid data) & $1.00\times 10^{-6}$ & 0.1479 & 0.1603 & 0.0985 & 0.0985 & $5.03\times 10^{-7}$ \\
3D (fine-grid data) & $1.00\times 10^{-7}$ & 0.1448 & 0.1685 & 0.0959 & 0.0959 & $5.03\times 10^{-6}$ \\
 
\bottomrule
\end{tabular}
\end{table}

Three observations. First, in 2D the heuristic rule adapts to the effective data error: at $\delta=10^{-8}$ it selects $k_*=2$ with an error of $0.0138$, attaining the oracle, whereas the discrepancy principle, fed the misleadingly small nominal $\delta$, over-iterates against accuracy the coarse model cannot deliver ($0.0510$). Second, in 3D all rules plateau at the level dictated by the mismatch; decreasing the nominal $\delta$ demonstrates nothing further, exactly as one expects when the model error dominates. Third, the residual ratio collapses ($\rho_{\mathrm{emp}}\ll1$): within the horizon, the finite-dimensional coarse operator can fit the smooth data almost exactly, so the residual is not floored by the model error; the effect of the mismatch appears in the reconstruction error, not in the residual. We emphasize, for honesty, that this collapse is driven primarily by the smoothness of the constant-direction perturbation rather than by the mismatch itself: the matched 2D/3D runs with the same perturbation exhibit equally small finite-window ratios (Subsection~\ref{subsec:robust}). The mismatch rows confirm that $\rho_{\mathrm{emp}}$ must not be read as a $\kappa$-diagnostic outside the matched setting, and both settings together illustrate how smooth, structured perturbations are the adversarial case for residual-based noise conditions (cf.\ the oscillatory row of Table~\ref{tab:noise}).

\subsection{Discussion}\label{subsec:discussion}

The picture is consistent across the three examples, which differ in dimension, in the smoothing operator (constant-coefficient, variable-coefficient, and non-self-adjoint convection--diffusion), and in the exact solution. The heuristic rule produces errors that decrease with the noise level, without any knowledge of $\delta$, and at moderate noise levels it is competitive with the best a~posteriori results. The rule \eqref{eq:dp} depends sensitively on $\tau$: the inadmissible $\tau=0.8$ systematically over-iterates, $\tau=1.9$ over-iterates at very small noise levels, and $\tau=12$ performs well overall, but a good $\tau$ cannot be known in advance. The weighted rule \eqref{eq:msrule} of \cite{mahale-singh} is accurate at small noise levels, consistent with its order-optimal theory, but at the largest noise levels it degenerates to $k_\delta=0$ (the threshold $\tau\delta$ already exceeds the initial weighted residual, whose size is governed by $\alpha_0\|x_0-x^\dg\|$ rather than by the noise); like the discrepancy principle, it thus relies on accurate knowledge of $\delta$.

Since the matched tables isolate the setting of the theorems, the small-noise behavior can be attributed correctly. In Tables~\ref{tab:results2d}--\ref{tab:results3d} the selected index $k_*$ and the surrogate $\delta_*$ eventually stagnate as $\delta$ decreases ($k_*=2$, $\delta_*\approx3.9\times10^{-6}$ in 2D, with the heuristic error settling near $0.021$ against an oracle of $0.0045$; $k_*=3$, $\delta_*\approx1.0\times10^{-6}$ in 3D, with an oracle gap widening to a factor of five at $\delta=10^{-8}$), even though the data is exact up to the added noise. We call this a \emph{finite-range parameter-choice plateau} (numerical stagnation of $k_*$ and $\delta_*$): $\theta$ settles at an index governed by the penalty term rather than by the by-then-negligible noise. This plateau is compatible with the weak upper bound of Corollary~\ref{cor:deltastar} but not predicted by it (the corollary supplies an upper bound and does not prove constancy of the selected index or residual), and it is distinct from the saturation of the convergence rate at $p=1$ (Section~\ref{sec:rates}), for which we reserve the word ``saturation''. The 1D runs show the same stagnation in milder form ($\delta_*$ settles near $10^{-5}$) while the error still decreases markedly ($0.79$ at $\delta=10^{-4}$ to $0.0039$ at $\delta=10^{-7}$, faster than $\sqrt\delta$ over this range). In particular, $\delta_*/\delta$ grows as $\delta$ decreases in all three examples, so we do not claim the regime $\delta_*=O(\delta)$ to be typical; the observed errors are consistent with the guaranteed bound $O(\sqrt{\delta_*}+\sqrt\delta)$. The mismatch stress tests of Subsection~\ref{subsec:stress} tell the complementary practical story: when the nominal $\delta$ underestimates the effective data error, the $\delta$-based rules over-iterate while the heuristic rule adapts, attaining the oracle at the two smallest noise levels of the 2D stress rows; this is the practical argument for noise-level-free rules, though it is deliberately not presented as confirmation of the $\delta\to0$ theory.

Two honest caveats remain. First, when $\delta$ (and a suitable $\tau$) are reliably known and the model is exact, the discrepancy principle can match or outperform the heuristic rule. Second, the guaranteed rate is only the baseline $O(\delta^{1/4})$, understood along admissible growing-horizon families (Remark~\ref{rem:growing-horizon}); the $O(\sqrt\delta)$ bound requires $\delta_*=O(\delta)$, a regime not observed asymptotically in any of our runs. The practical advantage of Rule~\ref{rule:heuristic} lies precisely where reliable noise information is unavailable.

\section{Conclusion}\label{sec:concl}

We analyzed the nonstationary iterated Lavrentiev method combined with a retrospective heuristic parameter-choice rule in the spirit of Hanke and Raus \cite{hanke-raus} and Jin and Wang \cite{jin-wang}, which terminates the iteration without knowledge of the noise level. For monotone operators with Lipschitz continuous, locally self-adjoint derivatives we proved well-definedness of the method and of the rule, the error estimate $O(\sqrt{\delta_*}+\sqrt{\delta})$ under the range condition $x_0-x^\dg\in\R(F'(x^\dg))$ for $k_*\ge1$, the two-sided bound $\kappa\delta\le\delta_*\le C\sqrt\delta$ (hence the baseline rate $O(\delta^{1/4})$ along admissible growing-horizon families, and $O(\sqrt\delta)$ whenever $\delta_*=O(\delta)$), and H\"older-rate results for $p\ge1$ including the saturation at $p=1$. Three questions remain open: the case $p<1$, the removal of the self-adjointness assumption on $F'$, and the completion of Proposition~\ref{thm:oracle-pre} to a full oracle inequality (Subsection~\ref{subsec:obstruction}). Compared with discrepancy-type rules \cite{mahale-nair09,mahale-nair13,mahale-singh}, the approach removes the need for precise noise information while retaining rigorous, if more limited, error estimates.

The numerical experiments in one, two, and three space dimensions, based on genuinely monotone operators $F(x)=A^*\varphi(Ax)$ with constant-coefficient, variable-coefficient, and non-self-adjoint convection--diffusion smoothing operators (whose derivatives are self-adjoint, in accordance with Assumption~\ref{ass:nonlinearity}(iv)), support the theoretical findings within their stated scope: in the independently recomputed matched-model runs the heuristic rule delivered stable reconstructions, with a finite-range plateau of $k_*$ and $\delta_*$ at the smallest noise levels that is compatible with (though not predicted by) Corollary~\ref{cor:deltastar}, and in the clearly separated model-mismatch stress tests, which lie outside the exact theory, it adapted to the effective data error where the $\delta$-based rules over-iterated. In relation to the author's earlier work \cite{mahale-singh-etna}, which analyzed a heuristic rule for this method under a different and in general more restrictive nonlinearity condition, the present paper shows that the heuristic approach remains viable for potential monotone operators under a Lipschitz condition on $F'$. Future directions include, beyond the three open problems above, general source conditions, Banach space settings, and parameter identification applications.

\appendix
\section{Horizon and window diagnostics}\label{app:horizon}

Table~\ref{tab:horizon} summarizes, for each matched-model experiment, the computational budget $K$, the ranges of the selected indices, the window diagnostics, and the status of the horizon conditions $K_\infty\ge\hat k$ and $K_\infty\ge k_*^\delta$; these involve constants not accessible from the data and are therefore assumed, not verified (Remarks~\ref{rem:retrospective} and~\ref{rem:growing-horizon}), and the diagnostics are consistent with, but do not certify, them.

\begin{table}[htb]
\centering
\caption{Horizon/window summary of the matched-model experiments (all noise levels of the respective table). ``$\theta\nearrow$ final third'': $\theta(k,y^\delta)$ strictly increasing over the last third of the window; ``$k_*$ stable'': $k_*$ unchanged when $K$ is enlarged by $8$. In the 3D example the residual decays proportionally to $\alpha_k$ beyond $k\approx5$, so $\theta$ is asymptotically constant with a shallow interior minimum ($\theta(K)/\theta(k_*)\in[1.004,1.18]$), yet $k_*$ remains unchanged under enlargement of $K$. The horizon conditions $K_\infty\ge\hat k$ (Theorem~\ref{thm:rate-range}) and $K_\infty\ge k_*^\delta$ (Corollary~\ref{cor:deltastar}) are assumed (not data-certifiable, cf.\ Remark~\ref{rem:growing-horizon}).}
\label{tab:horizon}
\footnotesize
\begin{tabular}{l c c c c c c}
\toprule
experiment & $K$ & $k_*$ (range) & $k_\delta$, $\tau{=}1.9$ (range) & $\theta\nearrow$ final third & $k_*$ stable & horizon cond.\\
\midrule
1D (Table~\ref{tab:results})    & 25 & 2--4 & 2--10 & yes & yes & assumed \\
2D (Table~\ref{tab:results2d})  & 24 & 1--2 & 1--10 & yes & yes & assumed \\
3D (Table~\ref{tab:results3d})  & 24 & 2--3 & 1--9  & no (asympt.\ const.) & yes & assumed \\
\bottomrule
\end{tabular}
\end{table}

\subsection*{Acknowledgment}
The authors gratefully acknowledge the infrastructural support provided by SVKM's College of Engineering, Shirpur, and the Indian Institute of Information Technology, Design and Manufacturing (IIITDM), Kancheepuram.

\subsection*{Funding}
No funding was received for this work.

\subsection*{Competing Interests}
The authors declare that they have no competing interests.

\end{document}